\documentclass[oneside,10pt]{amsart}
\usepackage{enumitem}
\usepackage{latexsym, bbm, amssymb, amsmath, wasysym, mathtools}
\usepackage[round]{natbib}
\usepackage{setspace}
\usepackage{paralist}
\usepackage{ifpdf}
\usepackage[capposition=top]{floatrow}

\usepackage{tikz}
\usetikzlibrary{patterns,snakes,shapes}

\usepackage{accents}

\makeatletter
\def\subsection{\@startsection{subsection}{1}%
  \z@{.5\linespacing\@plus.7\linespacing}{.3\linespacing}%
  {\normalfont\scshape}}
\newcommand{\eqnum}{\refstepcounter{equation}\textup{\tagform@{\theequation}}}
\makeatother

\DeclarePairedDelimiter\abs{\lvert\,}{\,\rvert}%
\DeclarePairedDelimiter\norm{\lVert\,}{\,\rVert}%

\theoremstyle{plain}
\newtheorem{theorem}{Theorem}
\newtheorem{corollary}[theorem]{Corollary}
\newtheorem{lemma}[theorem]{Lemma}
\newtheorem{proposition}[theorem]{Proposition}
\theoremstyle{definition}

\newtheorem{definition}[theorem]{Definition}

\newtheorem{remark}[theorem]{Remark}

\allowdisplaybreaks[1]

\newcommand{\R}{\mathbb{R}}

\renewcommand{\P}{\mathbb{P}}

\renewcommand{\S}{\mathcal{S}}
\newcommand{\E}{\mathbb{E}}
\newcommand{\F}{\mathcal{F}}

\newcommand{\G}{\mathcal{G}}

\newcommand{\A}{\mathcal{A}}
\newcommand{\I}{\mathcal{I}}
\renewcommand{\H}{\mathcal{H}}
\newcommand{\C}{\mathcal{C}}

\newcommand{\B}{\mathcal{B}}

\newcommand{\X}{\mathcal{X}}

\newcommand{\Kn}{K^{(n)}}

\newcommand{\1}{\mathbbm{1}}

\DeclareMathOperator{\Var}{Var}
\DeclareMathOperator{\Osc}{Osc}

\DeclareMathOperator{\esssup}{esssup}
\DeclareMathOperator{\essinf}{essinf}

\providecommand{\norm}[1]{\left\lVert#1\right\rVert}
\providecommand{\abs}[1]{\left\lvert#1\right\rvert}

\let\oldmarginpar\marginpar
\renewcommand\marginpar[1]{\-\oldmarginpar[\raggedleft\footnotesize #1]%
{\raggedright\footnotesize #1}}

\definecolor{mypurple}{rgb}{.3,0,.5}

\definecolor{lightyellow}{RGB}{255,255,102}

\begin{document}

\title[]%
{A Central Limit Theorem for \\ Temporally Non-Homogenous Markov Chains\\ with Applications to Dynamic Programming}
\author[]
{Alessandro Arlotto and J. Michael Steele}

\thanks{A. Arlotto:
The Fuqua School of Business, Duke University, 100 Fuqua Drive, Durham, NC, 27708.
Email address: \texttt{alessandro.arlotto@duke.edu}
}

\thanks{J. M. Steele:
Department of Statistics, The Wharton School, University of Pennsylvania, 3730 Walnut Street, Philadelphia, PA, 19104.
Email address: \texttt{steele@wharton.upenn.edu}
}



\begin{abstract}
    We prove a central limit theorem for a class of additive processes that arise naturally in the theory of
    finite horizon Markov decision problems. The main theorem generalizes a classic result of \citet{Dobrushin:TPA1956}
    for temporally non-homogeneous Markov chains, and the principal innovation is that here
    the summands are permitted to depend
    on both the current state and a bounded number of future states of the chain. We show through several examples that this added
    flexibility gives one a direct path to asymptotic normality of the optimal total reward of finite horizon
    Markov decision problems. The same examples also explain why such results are not easily obtained by
    alternative Markovian techniques such as enlargement of the state space.

    \medskip
    \noindent{\sc Mathematics Subject Classification (2010)}:
    Primary: 60J05, 90C40;
    Secondary: 60C05, 60F05, 60G42, 90B05, 90C27, 90C39.

    \medskip
    \noindent{\sc Key Words:}
    non-homogeneous Markov chain, central limit theorem,
    Markov decision problem, sequential decision, dynamic inventory management, alternating subsequence.

\end{abstract}

\date{first version: May 4, 2015; this version: December 6, 2015.
}

\maketitle



\section{Stochastic Dynamic Programs and Asymptotic Distributions} \label{se:intro}

In a finite horizon stochastic dynamic program (or Markov decision problem)
with $n$ periods, it is typical that the decision policy $\pi^*_n$ that maximizes total expected reward
will take actions that depend on both the current state of the system and on the number of periods that remain within the horizon.
The total reward $R_n(\pi^*_n)$ that is obtained when one follows the mean-optimal
policy $\pi^*_n$ will have the expected value that optimality requires, but the actual reward $R_n(\pi^*_n)$ that is realized
may --- or may not --- behave in a way that is well summarized by its expected value alone.

As a consequence, a well-founded judgement about the economic value of the policy $\pi^*_n$ will typically require
a deeper understanding of the random variable $R_n(\pi^*_n)$. One gets meaningful benefit from the knowledge of
the variance of $R_n(\pi^*_n)$ or its higher moments \citep{ArlottoGansSteele:OR2014},
but, in the most favorable instance, one would hope to know the
distribution of $R_n(\pi^*_n)$, or at least an asymptotic approximation to that distribution.

Limit theorems for the total reward (or the total cost)
of a Markov decision problem (or MDP) have been studied extensively, but
earlier work has focused almost exclusively on those problems where the optimal
decision policy is stationary. The first steps were taken
by \citet{Mandl:KYBER1973,Mandl:AAP1974,Mandl:PROCEEDING1974}
in the context of finite state space MDPs. This work was subsequently
refined and extended to more general MDPs by \citet{Mandl:SD1985}, \citet{MandlLausmanova:AOR1991},
\citet{Mendoza:MO2008}, and \citet{MendozaHernandez:JAP2010}.
Through these investigations one now has a substantial limit theory for a rich class of MDPs that
includes infinite-horizon MDPs with discounting and infinite horizon MDPs where one seeks to
maximize the long-run average reward.

Distributional properties of MDPs have also been considered in the
design of pathwise asymptotic optimal controls.
For instance, \citet{Leizarowitz:STO1987,Leizarowitz:STO1988}, \citet{Rotar:1985,Rotar:1986},
\citet{AsrievRotar:SSR1990}, \citet{Rotar:1991}, and \citet{BelkinaRotar:TVP2005}
studied controls that produce a long-run average reward that
is asymptotically optimal almost surely.
Also, \citet{Leizarowitz:MOR1996} investigates pathwise optimality in infinite horizon problems.
\citet{Rotar:DCD2012} provides a sustained review of this literature including a more comprehensive list of references.

Here the focus is on finite horizon MDPs and, to deal with such problems, one needs to break
from the framework of stationary decision policies. Moreover, for the purpose of the intended applications, it is
useful to consider additive functionals that are more complex than those that have been considered earlier in the theory
of temporally non-homogeneous Markov chains. These functionals are defined in the next subsection where we also give the statement of
our main theorem.

\subsection*{A Class of MDP Linked Processes}

In the theory of discrete-time finite horizon MDPs, one commonly studies a sequence of problems with increasing sizes.
Here, it will be convenient to consider two parameters, $m$ and $n$.  The parameter $m$ is fixed, and it will be determined by the nature of
the actions and rewards of the MDP. The parameter $n$ measures the size of the MDP; it is essentially the traditional horizon
size, but it comes with a small twist.

Now, for a given $m$ and $n$, we consider an arbitrary sequence of random variables
$\{ X_{n,i}: 1 \leq i \leq n+m\}$ with values in a Borel space $\X$, and we also consider an array of $n$ real valued functions
of $1+m$ variables,
$$
f_{n,i}: \X^{1+m} \rightarrow \R, \quad 1 \leq i \leq n.
$$
Further properties will soon be required for both the random variables
and the array of functions, but, for the moment, we only note that
the random variable of most importance to us here is the sum
\begin{equation}\label{eq:Sn}
S_n = \sum_{i=1}^n Z_{n,i}  \quad \text{where} \quad Z_{n,i}= f_{n,i}(X_{n,i}, \ldots, X_{n,i+m}).
\end{equation}

In a typical MDP application, the random variable $Z_{n,i}$ has an interpretation
as a reward for an action taken in period $i \in \{1, 2, \ldots,n\}$.  The size parameter $n$ is then
the number of periods in which decisions are made, and $S_n$ is the total reward received over all periods
$i \in \{1, 2, \ldots,n\}$ when one follows the policy $\pi_n$. Here, of course, the actions chosen by $\pi_n$
are allowed to depend on both the current \emph{time} and the current \emph{state}.

The parameter $m$ is new to this formulation, and, as we will shortly explain,  the flexibility provided by $m$ is precisely what
makes sums of the random variables $Z_{n,i}= f_{n,i}(X_{n,i}, \ldots, X_{n,i+m})$ useful in the theory of MDPs.
In the typical finite horizon setting, the index
$i$ corresponds to the decision period,
and the realized reward that is associated with period $i$ may depend on many things.
In particular, it commonly depends on $n$, $i$, the decision period state $X_{n,i}$,
and \emph{one or more} values of the post-decision period realizations of
the driving sequence $\{X_{n,i}: 1 \leq i \leq n+m\}$.

\subsection*{Requirements on the Driving Sequence}

We always require the driving
sequence $\{X_{n,i}: 1 \leq i \leq n+m\}$ to be a Markov process,
but here the Markov kernel for the transition between time $i$ and $i+1$  is allowed to change as $i$ changes.
More precisely, we take $\B(\X)$ to be the set of Borel subsets of the Borel space $\X$, and we define $\{X_{n,i}: 1 \leq i \leq n+m\}$
to be the \emph{temporally non-homogeneous Markov chain} that is determined by specifying a distribution for the initial value
$X_{n,1}$ and by making the transition from time $i$ to time $i+1$ in accordance with the Markov transition kernel
\begin{equation*}
\Kn_{i,i+1}(x,B) = \P(X_{n,i+1} \in B \,|\, X_{n,i} = x), \quad \text{where }  x \in \X \,\, \text{and }B \in \B(\X).
\end{equation*}

The transition kernels can be quite general, but we do require a condition on their minimal ergodic
coefficient. Here we first recall that for any
Markov transition kernel $K=K(x, dy)$ on $\X$, the \emph{\citeauthor{Dobrushin:TPA1956} contraction coefficient} is defined by
\begin{equation}\label{eq:contraction-coefficient-def}
    \delta(K) = \sup_{\substack{x_1, x_2 \in \X \\ B \in \B(\X)}} \abs{ K(x_1, B) - K(x_2, B) },
\end{equation}
and the corresponding \emph{ergodic coefficient} is given by
$$
\alpha(K) = 1 - \delta(K).
$$
Further, for an array $\{\Kn_{i,i+1} : 1 \leq i < n\}$ of Markov transition kernels on $\X$, the
\emph{minimal ergodic coefficient} of the $n$'th row is defined by setting
\begin{equation}\label{eq:minimal-ergodic-coeff}
\alpha_n = \min_{1 \leq i < n} \alpha(\Kn_{i,i+1}).
\end{equation}

There is also a minor technical point worth noting here. Although we study additive functionals
that can depend on the full row $\{X_{n,i}: 1 \leq i \leq n+m\}$ with $n+m$ elements,
the last $1+m$ elements of the row are used in a way that does not require any constraint on the associated
ergodic coefficients.
Specifically, the last $1+m$ elements of the row are used only to determine value of the time $n$
reward that one receives as a consequence of the last decision. It is for this reason that in expressions like
\eqref{eq:minimal-ergodic-coeff} we need only to consider $i$
in the range from $1$ to $n-1$.

\subsection*{Main Result: A CLT for Temporally Non-Homogeneous Markov Chains}

When the sums $\{S_n: n \geq 1\}$ defined by \eqref{eq:Sn} are centered and scaled, it is natural to expect that,
in favorable circumstances, they will converge in distribution to the standard Gaussian.
The next theorem confirms that this is the case provided that one has some modest
compatibility between the  size of
the minimal ergodic coefficient $\alpha_n$, the size of the functions $f_{n,i}$, $1 \leq i \leq n$, and the variance of
$S_n$.

\begin{theorem}[CLT for Temporally Non-Homogeneous Markov Chains]\label{th:CLT-nonhomogeneous-chain}
If there are constants $C_1, C_2,  \ldots$ such that
\begin{equation}\label{eq:asymptotic-condition-for-clt-1}
\max_{1 \leq i \leq n} \norm{f_{n,i}}_\infty  \leq C_n
\quad
\text{and}
\quad
C_n^2 \alpha_n^{-2}  = o(\Var[S_n]),
\end{equation}
then one has the convergence in distribution
\begin{equation}\label{eq:CLTLimit}
\frac{S_n - \E[S_n]}{\sqrt{\Var[S_n]}} \Longrightarrow N(0,1), \quad \quad \text{ as } n \rightarrow \infty.
\end{equation}
\end{theorem}

\begin{corollary}\label{cor:CLT}
If there are constants $c > 0$ and $C<\infty$ such that
$$\alpha_n \geq c \quad \text{and} \quad  C_n \leq C \, \text{ for all } n \geq 1,$$
then one has the asymptotic normality \eqref{eq:CLTLimit}
whenever $\Var[S_n] \rightarrow \infty$ as $n\rightarrow \infty$.
\end{corollary}

\begin{remark}[Boundedness Assumption]
One might hope to relax the condition in Theorem \ref{th:CLT-nonhomogeneous-chain}
that for each fixed $n \geq 1$ the functions $\{f_{n,i}: 1\leq i \leq n\}$ are uniformly bounded. Even though the oscillation bounds
in Section \ref{se:proof1-oscillation} make heavy use of the supremum norm, one could conceivably use truncation arguments
that still give access to effective oscillation bounds. Unfortunately, truncations would substantially complicate an argument that
is already long, so we have stayed with uniform boundedness.  In some simpler contexts, it is known that
the uniform boundedness condition
can be releaxed; specifically, there are such relaxations in the Markov additive CLTs  of
\citet{Nagaev:TVP1957,Nagaev:TVP1961}, \citet{Jon:PS2004}, and \citet{Statuljavicus:LMS1969}.
\end{remark}

\subsection*{Organization of the Analysis}

Before proving this theorem, it is useful to note
how it compares with the classic CLT of \citet{Dobrushin:TPA1956} for non-homogeneous Markov chains.
If we set $m=0$ in Theorem \ref{th:CLT-nonhomogeneous-chain} then we recover the
\citeauthor{Dobrushin:TPA1956} theorem, so the main issue is to understand how one benefits from
the possibility of taking $m\geq 1$. This is addressed in detail in Section \ref{se:Dobrushin} and in
the examples of Sections \ref{se:inventory} and \ref{se:alternating}.

After recalling some basic facts about the minimal ergodic coefficient in Section \ref{se:contractionCoefficient},
the proof begins in earnest in Section \ref{se:martingale} where
we note that there is a martingale that one can expect to be a good approximation for $S_n$.
The confirmation of the approximation is carried out in
Sections \ref{se:proof1-oscillation} and \ref{se:proof2-valueToGo}.
In Section \ref{se:proof3-end} we complete the proof by showing that the
assumptions of our theorem also imply that
the approximating martingale satisfies the conditions of a basic martingale central limit theorem.

We then take up applications and examples. In particular, we show in
Section \ref{se:inventory} that Theorem \ref{th:CLT-nonhomogeneous-chain}
leads to an asymptotic normal law for the optimal total cost
of a classic dynamic inventory management problem, and in Section \ref{se:alternating} we see how the theorem
can be applied to a well-studied problem in combinatorial optimization.

\section{On $m=0$ vs $m>0$ and \citeauthor{Dobrushin:TPA1956}'s CLT}\label{se:Dobrushin}

\citet{Dobrushin:TPA1956} introduced many of the concepts that are
central to the theory of additive functionals of a non-homogenous Markov chain.
In addition to introducing the contraction coefficient \eqref{eq:contraction-coefficient-def},
\citeauthor{Dobrushin:TPA1956} also provided one of the earliest --- yet most refined --- of the CLTs for
non-homogenous chains.

\begin{theorem}[\citealp{Dobrushin:TPA1956}]\label{th:CLT-Dobrushin}
If there are constants $C_1, C_2,  \ldots$ such that
\begin{equation}\label{eq:asymptotic-condition-for-clt-2}
\max_{1 \leq i \leq n} \norm{f_{n,i}}_\infty  \leq C_n
\quad
\text{and}
\quad
C_n^2 \alpha_n^{-3}  = o\bigg(\sum_{i=1}^n \Var[f_{n,i}(X_{n,i})] \bigg),
\end{equation}
then for $S_n = \sum_{i=1}^n f_{n,i}(X_{n,i})$
one has the asymptotic Gaussian law
$$
\frac{S_n - \E[S_n]}{\sqrt{\Var[S_n]}} \Longrightarrow N(0,1), \quad \quad \text{ as } n \rightarrow \infty.
$$
\end{theorem}

After \citeauthor{Dobrushin:TPA1956}'s work there were refinements and extensions by \citet{Sarymsakov:TPA1961},
\citet{Hanen:AIHP1963}, and \citet{Statuljavicus:LMS1969}, but the work that is closest to the approach taken here
is that of \citet{SetVar:EJP2005}. They used a martingale approximation
to give a streamlined proof of \citeauthor{Dobrushin:TPA1956}'s theorem, and they also used spectral theory to prove the
variance lower bound
\begin{equation}\label{eq:variance-lower-bound}
\frac{1}{4} \alpha_n\bigg( \sum_{i=1}^n \Var[f_{n,i}(X_{n,i})] \bigg) \leq \Var[S_n].
\end{equation}
This improves a lower bound of \citet[Theorem 1.2.7]{IosifescuTheodorescu:SPRI1969} by a factor of two, and
\citet[Corollary 15]{Pel:PTRF2012} gives some further refinements.

There are also upper bounds for the variance of $S_n$
in terms of the sum of the individual variances and
the reciprocal $\alpha_n^{-1}$ of the minimal ergodic coefficient. The most recent of these
are given by \citet{Sze:SPL2012}
where they are used in the analysis of continued fraction expansions among other things.

\subsection*{Comparison of Conditions}

Theorem \ref{th:CLT-nonhomogeneous-chain} requires that $C_n^2 \alpha_n^{-2}  = o(\Var[S_n])$ as $n \rightarrow \infty$
--- a condition that is directly imposed on the variance of the total sum $S_n$. On the other hand,
\citeauthor{Dobrushin:TPA1956}'s theorem imposes the condition \eqref{eq:asymptotic-condition-for-clt-2}
on the sum of the variances of the \emph{individual} summands.
This difference is not
accidental; it actually underscores a notable distinction between the traditional setting where $m=0$ and
the present situation where $m\geq 1$.

When one has $m=0$, the variance lower bound  \eqref{eq:variance-lower-bound} tells us that
condition \eqref{eq:asymptotic-condition-for-clt-2} of Theorem \ref{th:CLT-Dobrushin}
implies condition \eqref{eq:asymptotic-condition-for-clt-1} of Theorem \ref{th:CLT-nonhomogeneous-chain}, but,
when $m \geq 1$, there is not any analog to the lower bound \eqref{eq:variance-lower-bound}. This is
the nuance that forces us
to impose an explicit condition on the variance of the sum $S_n$ in  Theorem \ref{th:CLT-nonhomogeneous-chain}.

A simple example can be used to illustrate the point.
We take $m = 1$ and for each $n \geq 1$ we consider a sequence $X_{n,1},X_{n,2},\ldots, X_{n,n+1}$ of independent identically distributed
random variables with $0<\Var[X_{n,1}]<\infty$.
The minimal ergodic coefficient in this case is just $\alpha_n = 1$.
Next, for $1 \leq i \leq n$ we consider the function
$$
f_{n,i}(x,y) = \begin{cases}
\phantom{-}x & \text{if $i$ is even} \\
-y & \text{if $i$ is odd;}
\end{cases}
$$
we then set $S_0=0$, and, more generally, we let
\begin{equation*}
S_n = \sum_{i=1}^n f_{n,i}(X_{n,i}, X_{n,i+1}).
\end{equation*}
Now, for each $n \geq 0$ we see that cancellations in the sum give us
$S_{2n} = 0$ and $S_{2n+1} = - X_{2n+1,2(n+1)}$,
so, according to parity we find
$$
\Var[S_{2n}] = 0 \quad \quad \text{and} \quad \quad \Var[S_{2n+1}] = \Var[X_{n,1}].
$$
In particular, we have $\Var[S_{n}]=O(1)$ for all $n \geq 1$, while, on the other hand, for the
sum of the individual variances we have that
$$
\sum_{i=1}^n \Var[f_{n,i}(X_{n,i}, X_{n,i+1})] = n \Var[X_{n,1}] =\Omega(n).
$$
The bottom line is that when $m\geq 1$, there is no analog of the lower bound \eqref{eq:variance-lower-bound},
and, as a consequence, a result like Theorem \ref{th:CLT-nonhomogeneous-chain} needs to impose an explicit condition on
$\Var[S_{n}]$ rather than a condition on the sum of the variances of the individual summands.

\subsection*{Two Related Alternatives}

One might hope to prove Theorem \ref{th:CLT-nonhomogeneous-chain} by considering an enlarged state space
where one could first apply \citeauthor{Dobrushin:TPA1956}'s CLT (Theorem \ref{th:CLT-Dobrushin}) and then extract
Theorem \ref{th:CLT-nonhomogeneous-chain} as a consequence. For example, given the
conditions of Theorem \ref{th:CLT-nonhomogeneous-chain} with $m=1$, one might introduce the
bivariate chain $\{\widehat{X}_{n,i} = (X_{n,i},X_{n,i+1}): 1 \leq i \leq n\}$
with the hope of extracting the conclusion of
Theorem \ref{th:CLT-nonhomogeneous-chain} by applying \citeauthor{Dobrushin:TPA1956}'s theorem to $\{\widehat{X}_{n,i}: 1 \leq i \leq n\}$.

The fly in the ointment is that the resulting bivariate chain can be degenerate in the sense that the minimal ergodic coefficient
of the chain $\{\widehat{X}_{n,i}: 1 \leq i \leq n\}$ can equal zero. In such a situation,
\citeauthor{Dobrushin:TPA1956}'s theorem does not apply to the process $\{\widehat{X}_{n,i}: 1 \leq i \leq n\}$, even though Theorem \ref{th:CLT-nonhomogeneous-chain} may still provide a useful central limit theorem.
We give two concrete examples of this phenomenon in Sections \ref{se:inventory} and \ref{se:alternating}.

A further way to try to rehabilitate the possibility of using the bivariate
chain $\{\widehat{X}_{n,i}: 1 \leq i \leq n\}$ is to appeal
to theorems where the minimal ergodic coefficient $\alpha_n$ is replaced with some less fragile quantity. For example,
\citet{Pel:PTRF2012} has proved that one can replace $\alpha_n$ in \citeauthor{Dobrushin:TPA1956}'s theorem
with the maximal coefficient of correlation $\rho_n$.
Since one always has $\rho_n \leq \sqrt{1-\alpha_n}$, \citeauthor{Pel:PTRF2012}'s CLT is guaranteed to apply at least as widely
as \citeauthor{Dobrushin:TPA1956}'s CLT. Nevertheless, the examples of Sections \ref{se:inventory} and \ref{se:alternating} both show that
this refinement still does not help.

\section{On Contractions and Oscillations}\label{se:contractionCoefficient}

To prove Theorem \ref{th:CLT-nonhomogeneous-chain},
we need to assemble a few properties of the \citeauthor{Dobrushin:TPA1956} contraction coefficient.
Much more can be found in
\citet[Section 4.3]{Seneta:SPRI2006}, \citet[Section 4.2]{Winkler:SPRI2003}, or \citet[Chapter 4]{DelMoral:SPRI2004}.

If $\mu$ and $\nu$  are two probability measures, we write $\norm{ \mu - \nu}_{\rm TV}$ for the
total variation distance between $\mu$ and $\nu$. \citeauthor{Dobrushin:TPA1956}'s coefficient \eqref{eq:contraction-coefficient-def}
can then be written as
\begin{equation*}
\delta(K)
= \sup_{x_1,x_2 \in \X} \norm{ K(x_1, \cdot) - K(x_2, \cdot) }_{\rm TV},
\end{equation*}
and one always has $0 \leq \delta(K)\leq 1$. For any two Markov kernels $K_1$ and  $K_2$ on $\X$,
we also set
$$
(K_1K_2)(x,B) = \int K_1(x,dz) K_2(z,B),
$$
so $(K_1K_2)(x,B)$ represents the probability that one ends up in $B$ given that one starts at $x$ and takes two steps:
the first governed by the transition kernel $K_1$ and the second governed by the kernel $K_2$.
A crucial property of the \citeauthor{Dobrushin:TPA1956} coefficient $\delta$ is that one has the product inequality
\begin{equation}\label{eq:contraction-coeff-product-bound}
\delta(K_1K_2) \leq \delta(K_1) \delta(K_2).
\end{equation}

Now, given any array $\{\Kn_{i,i+1} :  1 \leq i < n\}$  of Markov kernels and any pair of times
$1 \leq i < j \leq n $, one can form the multi-step transition kernel
$$
\Kn_{i,j}(x,B) = (\Kn_{i,i+1} \Kn_{i+1,i+2} \cdots \Kn_{j-1,j})(x,B),
$$
and, as the notation suggests,
the kernel $\Kn_{i,i+1}$ can change as $i$ changes.
The product inequality \eqref{eq:contraction-coeff-product-bound}
and the definition of the minimal ergodic coefficient \eqref{eq:minimal-ergodic-coeff} then tell us
\begin{equation}\label{eq:contraction-coeff-bound-1}
\delta(\Kn_{i,j}) \leq  ( 1 - \alpha_n)^{j-i} \quad \quad \text{for all } 1 \leq i < j \leq n.
\end{equation}

\citeauthor{Dobrushin:TPA1956}'s coefficient can also be characterized by the action of the Markov kernel
on a natural function class.
First, for any bounded measurable function $h: \X \rightarrow \R$ we note that the operator
$$
(K h)(x) = \int K(x, dz)h(z),
$$
is well defined, and one also has that the oscillation of $h$
$$
\Osc(h) = \sup_{z_1,z_2\in \X} \abs{ h(z_1) - h(z_2) } < \infty.
$$
Now, if one sets $\H = \{ h: \Osc(h) \leq 1 \}$, then
the \citeauthor{Dobrushin:TPA1956} contraction coefficient  \eqref{eq:contraction-coefficient-def} has a second characterization,
\begin{equation*}
    \delta(K) = \sup_{\substack{x_1, x_2 \in \X \\ h \in \H }}  \abs{ (K h)(x_1) - (K h)(x_2) }.
\end{equation*}
This tells us in turn that for any Markov transition kernel $K$ on $\X$ and for any bounded measurable function $h:\X \rightarrow \R$,
one has the oscillation inequality
\begin{equation}\label{eq:oscillation-inequality}
\Osc ( K h ) \leq \delta(K) \Osc( h ).
\end{equation}

This bound is especially useful when it is applied to the multi-step kernel given by
$
\Kn_{i,j}=\Kn_{i,i+1} \Kn_{i+1,i+2} \cdots \Kn_{j-1,j}.
$
In this case, the oscillation inequality \eqref{eq:oscillation-inequality}
and the upper bound \eqref{eq:contraction-coeff-bound-1} combine to give us
\begin{equation}\label{eq:oscillation-inequality-multistep}
\Osc ( \Kn_{i,j} h ) \leq \delta(\Kn_{i,j}) \Osc( h ) \leq (1-\alpha_n)^{j-i} \Osc(h).
\end{equation}
This basic bound will be used many times in the analysis of Section \ref{se:proof1-oscillation}.

\section{Connecting a Martingale to $S_n$}\label{se:martingale}

Our proof of Theorem \ref{th:CLT-nonhomogeneous-chain} exploits a martingale approximation like the one used by
\citet{SetVar:EJP2005} in their proof of the \citeauthor{Dobrushin:TPA1956} central limit theorem.
Closely related plans have
been used by \citet{Gordin:DANSSSR1969},
\citet{KipnisVaradhan:CMP1986}, \citet{Kif:TAMS1998},
\citet{WuWoodroofe:AP2004}, \citet{GordinPeligrad:BE2011}, and \citet{Pel:PTRF2012}, but prior to
\citet{SetVar:EJP2005} the martingale approximation method
seems to have been used only for stationary processes.

Here we only need a basic version of the CLT for an array of martingale difference sequences (MDS) that we
frame as a proposition. This version is easily covered by any of the martingale central limit theorems of
\citet{Bro:AMS1971}, \citet{McLei:AP1974}, or \citet[Corollary 3.1]{HalHey:AP1980}.

\begin{proposition}[Basic CLT for MDS Arrays]\label{pr:HallHeyde-MartingaleCLT}

If for each $n \geq 1$, one has a martingale difference sequence $\{ \xi_{n,i}: 1 \leq i \leq n\}$ with respect to
the filtration
$\{\G_{n,i} : 0 \leq i \leq n\}$,
and if one also has the  negligibility condition
\begin{equation}\label{eq:asymtotic-negligibility}
\max_{1 \leq i \leq n} \norm{ \, \xi_{n,i}\, }_\infty \longrightarrow 0  \quad \quad \text{ as } n \rightarrow \infty,
\end{equation}
then the ``weak  law of large numbers" for the conditional variances
\begin{equation}\label{eq:LLN-conditional-variances}
\sum_{i=1}^n \E[ \xi_{n,i}^2 \,|\, \G_{n,i-1}]  \stackrel{ p}{\longrightarrow}  1 \quad \quad \text{ as } n \rightarrow \infty,
\end{equation}
implies that one has convergence in distribution to a standard normal,
$$
\sum_{i=1}^n \xi_{n,i}\Longrightarrow N(0,1) \quad \quad \text{  as } n \rightarrow \infty.
$$
\end{proposition}

\subsection*{A Martingale for a Non-Homogenous Chain}

We let $\F_{n,0}$ be the trivial $\sigma$-field, and
we set $\F_{n,i} = \sigma\{X_{n,1}, X_{n,2}, \ldots, X_{n,i}\}$ for $1 \leq i \leq n+m$.
Further, we define the \emph{value to-go process} $\{V_{n,i}: m \leq i \leq n+m\}$
by setting $V_{n,n+m}=0$ and by letting
\begin{equation}\label{eq:value-togo}
V_{n,i} =  \sum_{j=i+1-m}^{n}  \E [Z_{n,j}  \,|\, \F_{n,i} ], \quad \quad \text{for } m \leq i < n+m.
\end{equation}

If we view the random variable $Z_{n,j}$ as a reward that we receive at time $j$, then the value to-go $V_{n,i}$ at time $i$
is the conditional expectation at time $i$ of the total of the rewards that
stand to be collected during the time interval $\{i+1-m, \ldots, n\}$.
For $1 + m \leq i \leq n + m $ we then let
\begin{equation}\label{eq:MDS}
d_{n,i}  = V_{n,i} - V_{n,i-1} +Z_{n,i-m},
\end{equation}
and one can check directly from the definition that
$\{d_{n,i}: 1 + m \leq i \leq n + m \} $ is a martingale difference sequence (MDS)
with respect to its natural filtration $\{ \F_{n,i}: 1 + m \leq i \leq n + m \}$.

When we sum the terms of \eqref{eq:MDS}, the summands $V_{n,i} - V_{n,i-1}$ telescope, and
we are left with the basic decomposition
\begin{equation}\label{eq:Sn-decomposition}
S_n = \sum_{i=1}^n Z_{n,i} =  V_{n,m}+ \sum_{i=1+m}^{n+m}  d_{n,i} .
\end{equation}
For the proof of Theorem \ref{th:CLT-nonhomogeneous-chain}, we assume without loss of generality that
$\E[Z_{n,i}]=0$ for all $1 \leq i \leq n$.
Naturally, in this case we also have $\E[S_n]=\E[V_{n,m}]=0$ since the sum of the martingale differences
in \eqref{eq:Sn-decomposition} will always have total expectation zero. We now just need to analyze the
components of the representation \eqref{eq:Sn-decomposition}.

\section{Oscillation Estimates}\label{se:proof1-oscillation}

The first step in the proof of Theorem \ref{th:CLT-nonhomogeneous-chain} is to argue that the summand $V_{n,m}$
in \eqref{eq:Sn-decomposition} makes a contribution to $S_n$ that is asymptotically negligible when compared to the
standard deviation of $S_n$.  Once this is done,
one can use the martingale CLT to deal with the last sum in \eqref{eq:Sn-decomposition}. Both of these
steps depend on oscillation estimates that exploit the multiplicative bound \eqref{eq:oscillation-inequality-multistep} on
the \citeauthor{Dobrushin:TPA1956} contraction coefficient.

For any random variable $X$ one has the trivial bound
\begin{equation}\label{eq:OscTrivial}
\Osc(X)=\esssup( X )-\essinf( X ) \leq 2 \norm{X}_\infty,
\end{equation}
together with its partial converse,
\begin{equation}\label{eq:OscTrivial2}
\norm{X-\E[X]}_\infty \leq \Osc(X).
\end{equation}
Moreover for any two $\sigma$-fields $\I \subseteq \I'$ of the Borel sets $\B (\X)$,
the conditional expectation is a contraction for the oscillation semi-norm; that is, one has
\begin{equation}\label{ex:CondExpOsc}
\Osc ( \E[X \,| \, \I] ) \leq \Osc ( \E[X \,| \, \I'] ) \leq \Osc (X).
\end{equation}
Also, by comparison of $X(\omega)Y(\omega)$ and $X(\omega')Y(\omega')$, one has the product rule
\begin{equation}\label{eq:productrule}
\Osc(XY) \leq \norm{X}_\infty \Osc(Y) +\norm{Y}_\infty \Osc(X).
\end{equation}

In the next two lemmas we assume
that there is a constant $C_n < \infty$ such that
\begin{equation*}
\norm{ f_{n,i} }_\infty \leq C_n \quad \quad \text{for all } 1 \leq i \leq n.
\end{equation*}
Since $Z_{n,i} = f_{n,i}(X_{n,i}, \ldots, X_{n,i+m})$ and $\E[Z_{n,i}]= 0$, this assumption gives us
\begin{equation}\label{eq:LinftyOscBoundsZ}
\norm{Z_{n,i}}_\infty \leq C_n, \quad \text{and} \quad \Osc (\E[Z_{n,i} \,| \, \I] )  \leq 2 C_n
\end{equation}
for any $\sigma$-field $\I \subseteq \B(\X)$.

\subsection*{Oscillation Bounds on Conditional Moments}

\begin{lemma}[Conditional Moments]\label{lm:OscBounds}
For all
$1 \leq i < j \leq n$ one has
\begin{equation}\label{eq:Linfty-bound}
\norm{ \E[Z_{n,j} \, | \, \F_{n,i}] }_\infty
    \leq \Osc( \E[Z_{n,j} \, | \, \F_{n,i}] )
   \leq 2 C_n (1 - \alpha_n)^{j-i},
\end{equation}
and
\begin{equation}\label{eq:Oscillation-bound-squared}
\Osc(  \E[ Z_{n,j}^2 \,|\, \F_{n,i}] )  \leq 2 C_n^2 (1 - \alpha_n)^{j-i}.
\end{equation}
\end{lemma}

\begin{proof}
Since $\E[Z_{n,j} \, | \, \F_{n,i}]$ has mean zero,
the first inequality of \eqref{eq:Linfty-bound} is immediate from \eqref{eq:OscTrivial2}.
To get the
second inequality, we note by the Markov property that we can
define a function  $h_j$  on the support of $X_{n,j}$ by setting
\begin{equation*}
h_j(X_{n,j}) =\E[Z_{n,j}\, | \, \F_{n,j}],
\end{equation*}
and by \eqref{eq:LinftyOscBoundsZ} we have the bound $\Osc(h_j) \leq 2C_n$.
For $i< j$ a second use of the Markov property gives us the pullback identity
$$
\E[Z_{n,j} \, | \, \F_{n,i}] = (\Kn_{i,j}h_j)(X_{n,i}),
$$
so the bound \eqref{eq:oscillation-inequality-multistep} gives us
$$
\Osc(\Kn_{i,j}  h_j) \leq 2 C_n ( 1 - \alpha_n)^{j-i},
$$
and this is all we need to complete the proof of \eqref{eq:Linfty-bound}.

One can prove \eqref{eq:Oscillation-bound-squared} by essentially the same method, but
now we define a map $x \mapsto s_j(x)$ by setting
\begin{equation*}
s_j(X_{n,j}) = \E[Z_{n,j}^2\, | \, \F_{n,j}],
\end{equation*}
so for $i < j$ the pullback identity becomes
$$
\E[Z_{n,j}^2 \, | \, \F_{n,i}] = (\Kn_{i,j}s_j)(X_{n,i}).
$$
By \eqref{ex:CondExpOsc} we have $\Osc(s_j) \leq  \Osc(Z_{n,j}^2)$, so \eqref{eq:LinftyOscBoundsZ}
implies $\Osc(s_j) \leq 2 C_n^2$, and the
inequality \eqref{eq:oscillation-inequality-multistep} then gives us \eqref{eq:Oscillation-bound-squared}.
\end{proof}

\subsection*{Oscillation Bounds on Conditional Cross Moments}

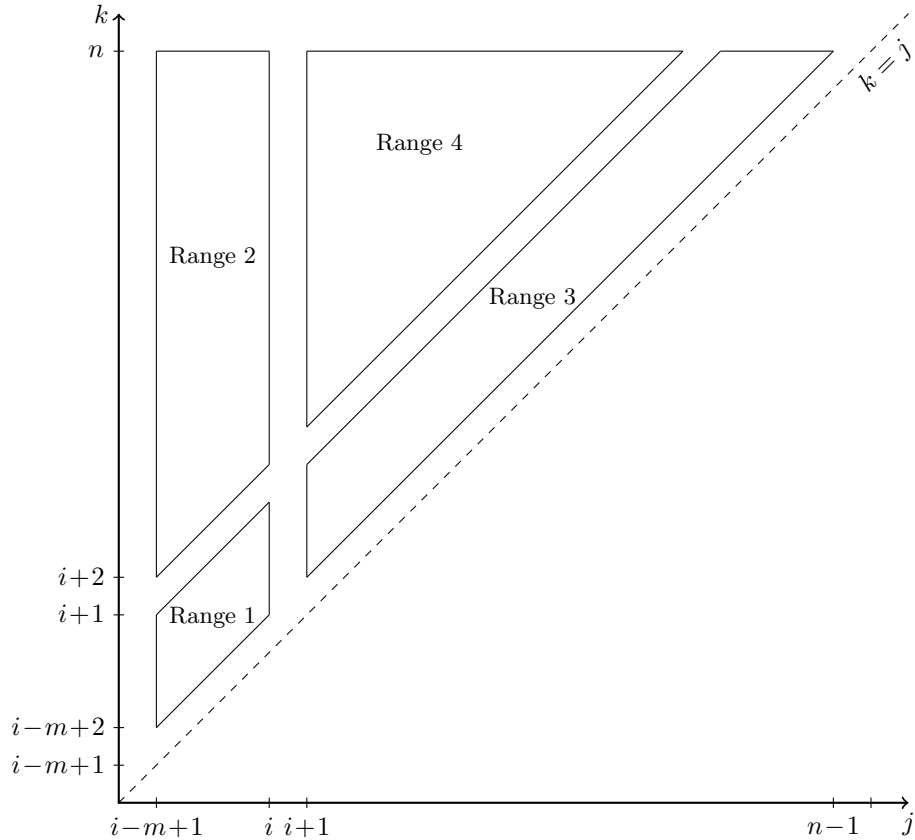
\begin{figure}[t]
\caption{Cross Moments Index Relations}\label{fig:indices}
\begin{tikzpicture}[y=.5cm, x=.5cm]
    \centering
     \draw [<->,thick] (0,21) node (yaxis) [left] {$k$}
        |- (21,0) node (xaxis) [below] {$j$};
    	
     	\draw (1,2pt) -- (1,-2pt)
     		node[anchor=north] {$i\!-\!m\!+\!1$};

        \draw (4,2pt) -- (4,-2pt)
     		node[anchor=north] {$i$};

    	\draw (5,2pt) -- (5,-2pt)
     		node[anchor=north] {$i\!+\!1$};

        \draw (19,2pt) -- (19,-2pt)
     		node[anchor=north] {$n\!-\!1$};

        \draw (20,2pt) -- (20,-2pt);


       	\draw (2pt,1) -- (-2pt,1)
     		node[anchor=east] {$i\!-\!m\!+\!1$};

     	\draw (2pt,2) -- (-2pt,2)
     		node[anchor=east] {$i\!-\!m\!+\!2$};

        \draw (2pt,5) -- (-2pt,5)
     		node[anchor=east] {$i\!+\!1$};

    	\draw (2pt,6) -- (-2pt,6)
     		node[anchor=east] {$i\!+\!2$};

        \draw (2pt,20) -- (-2pt,20)
     		node[anchor=east] {$n$};

    \draw (1,2)  -- (4,5) ;
    \draw (1,5)  -- (4,8) ;
    \draw (1,2)  -- (1,5) ;
    \draw (4,5)  -- (4,8) ;
    \node [anchor=south] at (2.5, 4.4) {\small Range 1};

    \draw (1,6)  -- (4,9) ;
    \draw (1,20)  -- (4,20) ;
    \draw (1,6)  -- (1,20) ;
    \draw (4,9)  -- (4,20) ;
    \node [anchor=south] at (2.5, 14) {\small Range 2};

    \draw (5,6)  -- (19,20) ;
    \draw (5,9)  -- (16,20) ;
    \draw (16,20)-- (19,20);
    \draw (5,6)  -- (5,9) ;
    \node [anchor=south] at (11, 12.9) {\small Range 3};

    \draw (5,10)  -- (15,20) ;
    \draw (5,20)  -- (15,20) ;
    \draw (5,10)  -- (5,20) ;
    \node [anchor=south] at (8, 17) {\small Range 4};

    \draw[dashed] (0,0) -- (21,21)
        node[rotate=45,anchor=north] at (20,20) {$k=j$ };

\end{tikzpicture}
\footnotetext{ The estimates  in Lemma \ref{lm:oscillation-bounds} require attention to certain ranges of indices.
In turn, these amount to a decomposition of the lattice triangle defined by the upper-left half of
$\{1,2, \ldots,n\}\times \{1,2, \ldots,n\}$.}
\end{figure}

The minimal ergodic coefficient $\alpha_n$ can also be used to
control the oscillation of the conditional expectations of the products $Z_{n,j}  Z_{n,k}$
given $\F_{n,i}$.
All of the inequalities that we need tell a similar story,
but the specific bounds have an inescapable dependence on the
relative values of $i$, $j$, $k$, $n$, and $m$.
Figure \ref{fig:indices} gives a graphical representation
of the constraints on the indices that feature in the next lemma.

\begin{lemma}[Conditional Cross Moments]\label{lm:oscillation-bounds}
For each $i \in \{m, \ldots, n+m\}$ we consider $i -m < j < n$ and $j < k \leq n$.
We then have the following oscillation bounds that depend
on the range of the indices (see also Figure \ref{fig:indices}):

\textsc{Range 1.} If $  j \leq i$ and $ k \leq j+m$ then
    \begin{equation}\label{eq:oscillation-bound-both-inside}
    \Osc( \E[Z_{n,j}  Z_{n,k}  \,|\,  \F_{n,i}] )
         \leq 4C_n^2 .
    \end{equation}

\textsc{Range 2.} If $ j \leq i  < j+m < k $ then
    \begin{equation}\label{eq:oscillation-bound-i-inside}
    \Osc( \E[Z_{n,j}  Z_{n,k}  \,|\,  \F_{n,i}] )
         \leq 6 C_n^2 (1 - \alpha_n)^{k- j - m}.
    \end{equation}

\textsc{Range 3.} If $ i < j < k \leq j+m$ then
    \begin{equation}\label{eq:oscillation-bound-k-inside}
    \Osc( \E[Z_{n,j}  Z_{n,k}  \,|\,  \F_{n,i}] )
         \leq 2 C_n^2 (1 - \alpha_n)^{j- i}.
    \end{equation}

\textsc{Range 4.} If $i < j \leq j+m < k$, then
    \begin{equation} \label{eq:oscillation-bound-both-outside}
    \Osc( \E[Z_{n,j}  Z_{n,k}  \,|\,  \F_{n,i}] )
         \leq 6 C_n^2 (1 - \alpha_n)^{k- i - m}.
    \end{equation}
\end{lemma}

\begin{proof}
Inequality \eqref{eq:oscillation-bound-both-inside} follows immediately
from the product rule \eqref{eq:productrule} and the bounds \eqref{eq:LinftyOscBoundsZ}.
To prove \eqref{eq:oscillation-bound-i-inside}, we note that for $i < j+m$ we have $\F_{n,i}\subseteq \F_{n,j+m}$
so from the monotonicity \eqref{ex:CondExpOsc}
and the fact that $Z_{n,j}$ is $\F_{n,j+m}$-measurable,
we obtain that
$$
\Osc( \E [ Z_{n,j} Z_{n,k} \, | \, \F_{n,i}]) \leq \Osc( \E [ Z_{n,j}  Z_{n,k} \, | \, \F_{n,j+m} ])
= \Osc(  Z_{n,j} \E [ Z_{n,k} \, | \, \F_{n,j+m} ]).
$$
The product rule \eqref{eq:productrule} applied to the quantity on the right-hand side above gives us the inequality
\begin{align*}
\Osc( \E [ Z_{n,j} & Z_{n,k} \, | \, \F_{n,i}]) \\
&\leq \norm{Z_{n,j}}_\infty  \Osc( \E [ Z_{n,k} \, | \, \F_{n,j+m} ])
+ \Osc (Z_{n,j})\norm{ \E [ Z_{n,k} \, | \, \F_{n,j+m} ]}_\infty,
\end{align*}
so if we recall that $\norm{ Z_{n,i}}_\infty \leq C_n$ and that $\Osc(Z_{n,j}) \leq 2 C_n$
and use the conditional moment bounds in \eqref{eq:Linfty-bound} we have
\begin{equation*}
\Osc( \E [ Z_{n,j}  Z_{n,k} \, | \, \F_{n,i}])  \leq 2 C_n^2 (1-\alpha_n)^{k-j-m} + 4 C_n^2 (1-\alpha_n)^{k-j-m},
\end{equation*}
completing the proof of \eqref{eq:oscillation-bound-i-inside}.

To verify inequality \eqref{eq:oscillation-bound-k-inside}, we consider the map $X_{n,j} \mapsto p_j(X_{n,j})$ given by
$$
p_j(X_{n,j}) = \E[Z_{n,j} Z_{n,k} \, | \, \F_{n,j}],
$$
and we note that for  $i < j$ we have the pullback identity
\begin{equation*}
\E[Z_{n,j} Z_{n,k} \, | \, \F_{n,i}]  = (\Kn_{i,j}p_j)(X_{n,i}).
\end{equation*}
Since $\norm{ Z_{n,j} }_\infty$ and $\norm{ Z_{n,k} }_\infty$ are bounded by $C_n$, we have
$\norm{ p_j }_\infty \leq C_n^2$ and $\Osc( p_j ) \leq 2 C_n^2$.
We also have  $i < j < k$ so \eqref{eq:oscillation-inequality-multistep}
tells us that
$$
\Osc(\Kn_{i,j}p_j) \leq \delta(\Kn_{i,j})\Osc(p_j) \leq 2C_n^2 (1-\alpha_n)^{j-i},
$$
completing the proof of \eqref{eq:oscillation-bound-k-inside}.

Finally, for the last inequality \eqref{eq:oscillation-bound-both-outside}
we have $ j \leq j+m <k$,
we consider the map $X_{n,j} \mapsto q_j(X_{n,j})$ defined by setting
$$
q_j(X_{n,j})= \E[ Z_{n,j} (\E[ Z_{n,k} \,| \, \F_{n, j+m}]) \, | \, \F_{n,j} ],
$$
and we obtain the identity
$$
\E[ Z_{n,j} Z_{n,k} \, | \, \F_{n,i} ] = (\Kn_{i,j}q_j)(X_{n,i}).
$$
By the multiplicative bound \eqref{eq:oscillation-inequality-multistep}, this gives us
$$
\Osc(\E[ Z_{n,j} Z_{n,k} \, | \, \F_{n,i} ] ) \leq (1-\alpha_n)^{j-i} \Osc(q_j),
$$
and we also have  $\Osc(q_j) \leq 6 C_n^2 (1 - \alpha_n)^{k-j-m}$ by \eqref{eq:oscillation-bound-i-inside},
so the proof of \eqref{eq:oscillation-bound-both-outside} is also complete.
\end{proof}

\section{The Value To-Go Process and MDS $L^\infty$-Bounds}\label{se:proof2-valueToGo}

We have everything we need to argue that the
variance condition \eqref{eq:asymptotic-condition-for-clt-1} implies the
negligibility condition \eqref{eq:asymtotic-negligibility}.
The first step is to get simple $L^\infty$-estimates of the value to-go  $V_{n,i}$
that was defined in \eqref{eq:value-togo}. We then need estimates of
the martingale difference$d_{n,i}$ defined in \eqref{eq:MDS}.
Here, and subsequently,  we use $M = M(m)$ to denote a Hardy-style constant which depends only on $m$
and which may change from one line to the next.

\begin{lemma}[$L^\infty$-Bounds for the Value To-Go and for the MDS]\label{lm:asymtotic-negligibility}
There is a constant $M< \infty$ such that for all $n\geq 1$ we have
\begin{align}
\norm{ V_{n,i} }_ \infty &\leq M C_n \alpha_n^{-1},
\quad \quad \text{ for } m \leq i \leq n+m,\quad \text{and}  \label{eq:Linfty-bound-value-functions}\\
\norm{ d_{n,i} }_\infty &\leq M C_n \alpha_n^{-1}, \quad \quad \text{ for } 1+m \leq i \leq n+m. \label{eq:Linfty-bound-MDS}
\end{align}
\end{lemma}

\begin{proof}
We have $\norm{Z_{n,j}}_\infty \leq C_n$, and when we use this estimate on the
first $m$ summands in the definition \eqref{eq:value-togo} of the value to-go $V_{n,i}$ we get the bound
$$
\norm{ V_{n,i} }_\infty \leq m \, C_n + \sum_{j=i+1}^{n} \norm{ \E [Z_{n,j}  \,|\, \F_{n,i} ] }_\infty.
$$
From \eqref{eq:Linfty-bound} we know that
$\norm{ \E [Z_{n,j}  \,|\, \F_{n,i} ] }_\infty \leq 2 C_n (1-\alpha_n)^{j-i}$
for all $1 \leq i < j \leq n$ so, after completing the geometric series, we have
$$
\norm{ V_{n,i} }_\infty \leq m \, C_n + 2 C_n \alpha_n^{-1} \leq M C_n \alpha_n^{-1},
$$
where one can take $M = 2m$ as a generous choice for $M$.
This bound, the representation \eqref{eq:MDS}, and the triangle inequality then
give us \eqref{eq:Linfty-bound-MDS}.
\end{proof}


\subsection*{Conditional Variances $L^2$-Bounds}

Everything is also in place to show that the variance condition \eqref{eq:asymptotic-condition-for-clt-1}
gives one the weak law of large numbers for the conditional variances
\eqref{eq:LLN-conditional-variances}.
We begin by deriving some basic inequalities for the variance of $S_n$.

\begin{lemma}[Variance Bounds]
For all $n \geq 1$ we have
\begin{equation}\label{eq:ExpSnSq}
\E[ S_n^2]=\E[ V_{n,m}^2] + \sum_{j=1+m}^{n+m} \E[d_{n,j}^2], \quad \text{and}
\end{equation}
\begin{equation}\label{eq:variance-mds-bounds}
\Var[S_n] - M C_n^2 \alpha_n^{-2} \leq \sum_{j=1+m}^{n+m} \E[  d_{n,j}^2] \leq \Var[S_n].
\end{equation}
\end{lemma}
\begin{proof} When we square both sides of \eqref{eq:Sn-decomposition} we have
$$
S_n^2 = V_{n,m}^2 +  2 V_{n,m} \bigg\{\sum_{j=1+m}^{n+m} d_{n,j}\bigg\} + \bigg\{\sum_{j=1+m}^{n+m} d_{n,j}\bigg\}^2.
$$
Since $V_{n,m}$ is $\F_{n,m}$-measurable, we obtain from the conditional orthogonality of the martingale differences that
$$
\E[S_n^2 \, | \, \F_{n,m}] = V_{n,m}^2 + \sum_{j=1+m}^{n+m} \E[ d_{n,j}^2 \, | \, \F_{n,m}],
$$
and, when we take the total expectation, we then get  \eqref{eq:ExpSnSq}.
Finally, since $\E[ S_n]=0$, the representation \eqref{eq:ExpSnSq} and
the bound \eqref{eq:Linfty-bound-value-functions} for $\norm{V_{n,m}}_\infty$
give us the two inequalities of \eqref{eq:variance-mds-bounds}.
\end{proof}

\begin{lemma}[Oscillation Bound]\label{lm:oscillation-mds}
There is a constant $M < \infty$ such that
\begin{equation}\label{eq:oscillation-bound-sum-dtsquared}
\Osc( \sum_{j = 1+i}^{n+m} \E[  d_{n,j}^2 \,|\, \F_{n,i}] ) \leq M C_n^2 \alpha_n^{-2}
\quad \quad \text{for } m \leq i \leq n+m.
\end{equation}
\end{lemma}

\begin{proof}
If we sum the identity \eqref{eq:MDS} we have
$$
\sum_{j=1+i}^{n+m} Z_{n,j-m}= V_{n,i} +  \sum_{j=1+i}^{n+m} d_{n,j},
$$
so, when we square both sides and use the fact that $V_{n,i}$ is $\F_{n,i}$-measurable, the orthogonality of the martingale differences
gives us
$$
\E\big[\big\{\sum_{j=1+i}^{n+m} Z_{n,j-m}  \big\}^2 \, | \, \F_{n,i}  \big] =V_{n,i}^2 + \sum_{j=i+1}^n \E[d_{n,j}^2 \, |\, \F_{n,i}].
$$
The triangle inequality then implies
\begin{equation}\label{eq:oscillation-bound-mds-squared}
\Osc \big( \sum_{j = i+1}^{n+m} \E[ d_{n,j}^2 \,|\, \F_{n,i}] \big)
 \leq  \Osc( V_{n,i}^2 ) + \Osc \big(  \E\big[\big\{\sum_{j=1+i}^{n+m} Z_{n,j-m}  \big\}^2 \, | \, \F_{n,i}  \big]  \big) .
\end{equation}
By \eqref{eq:Linfty-bound-value-functions} we have $\norm{ V_{n,i} }_ \infty \leq M C_n \alpha_n^{-1}$  so, by \eqref{eq:OscTrivial},
we obtain
\begin{equation}\label{eq:oscillation-bound-vi-squared}
\Osc(V_{n,i}^2) \leq 2 \norm{ V_{n,i}^2 }_ \infty \leq M C_n^2 \alpha_n^{-2}.
\end{equation}

It only remains to estimate the second summand of \eqref{eq:oscillation-bound-mds-squared}, but this takes some work.
Specifically, we will check that one can write
\begin{equation}\label{eq:oscillation-decomposition-S1S2S3S4S5}
 \Osc (  \E [ \big\{\!\!\! \sum_{j=1+i-m}^n  Z_{n,j}\, \big\}^2 | \F_{n,i} ] )  \leq \S_0 + \S_1 + \S_2 + \S_3 + \S_4.
\end{equation}
where $\S_0, \S_1, \S_2, \S_3$, and $\S_4$ are non-negative sums that one can estimate individually with help from our oscillation bounds.
Here the first term $\S_0$ accounts for the oscillation of the conditional squared moments. It is given by
$$
\S_0 = \sum_{j=1+i-m}^i \Osc (  \E [   Z_{n,j}^2 | \F_{n,i} ] ) + \sum_{j=1+i}^n \Osc (  \E [   Z_{n,j}^2 | \F_{n,i} ] ),
$$
and by \eqref{eq:LinftyOscBoundsZ} and \eqref{eq:Oscillation-bound-squared} we have the estimate
$$
\S_0 \leq 2 m C_n^2 +  2 C_n^2 \sum_{j=1+i}^n (1-\alpha_n)^{j-i} \leq 2(1+m) C_n^2 \alpha_n^{-1}.
$$

The remaining sums $\S_1, \S_2, \S_3$ and $\S_4$ are given by the oscillation
of the conditional cross moments $Z_{n,j}Z_{n,k}$ given $\F_{n,i}$ where
the ranges of the indices $j$ and $k$ are given by the corresponding four regions in Figure \ref{fig:indices}.
Specifically, we have
$$
\S_1 = 2 \sum_{j=1+i-m}^{i} \sum_{k=1+j}^{j+m} \Osc ( \E[  Z_{n,j}  Z_{n,k} \,|\, \F_{n,i} ] ),
$$
and \eqref{eq:oscillation-bound-both-inside} gives us
$\S_1 \leq 8 m^2 C_n^2$
since $\S_1$ has $m^2$ summands.
Next, if we set
$$
\S_2 = 2\!\!\! \sum_{j=1+i-m}^{i} \sum_{k=1+j+m}^{n} \Osc ( \E[  Z_{n,j}  Z_{n,k} \,|\, \F_{n,i} ] )
$$
then the oscillation inequality \eqref{eq:oscillation-bound-i-inside} gives us
$$
\S_2 \leq 12 C_n^2 \sum_{j=1+i-m}^{i} \sum_{k=1+j+m}^{n} (1-\alpha_n)^{k-j-m}
     \leq 12 m C_n^2 \alpha_n^{-1}.
$$
Similarly, for the third region, the bound \eqref{eq:oscillation-bound-k-inside} gives us
\begin{align*}
\S_3 &= 2 \sum_{j=1+i}^{n} \sum_{k=1+j}^{j+m} \Osc (  \E [ Z_{n,j} Z_{n,k} \,|\, \F_{n,i} ] ) \\
&\leq 4 C_n^2 \sum_{j=1+i}^{n} \sum_{k=1+j}^{j+m} (1-\alpha_n)^{j-i}
     \leq 4 m C_n^2 \alpha_n^{-1},
\end{align*}
and, for the fourth region, the bound \eqref{eq:oscillation-bound-both-outside} implies
\begin{align*}
\S_4 &= 2 \sum_{j=1+i}^{n} \sum_{k=1+j+m}^{n} \Osc (  \E [ Z_{n,j} Z_{n,k} \,|\, \F_{n,i} ] ) \\
&\leq 12 C_n^2 \sum_{j=1+i}^{n} \sum_{k=1+j+m}^{n}  (1-\alpha_n)^{k-i-m}  \leq 12 C_n^2 \alpha_n^{-2}.
\end{align*}
Finally, by our decomposition \eqref{eq:oscillation-decomposition-S1S2S3S4S5}, the upper bounds for $\S_0, \S_1, \S_2,\S_3$, and $\S_4$
tell us that there is a constant $M$ for which we have
$$
 \Osc (  \E [ \big\{\!\!\! \sum_{j=1+i-m}^n  Z_{n,j}\, \big\}^2 | \F_{n,i} ] ) \leq M C_n^2 \alpha_n^{-2},
$$
so, given \eqref{eq:oscillation-bound-mds-squared} and \eqref{eq:oscillation-bound-vi-squared}, the proof of the lemma is complete.
\end{proof}

\section{Completion of the Proof of Theorem \ref{th:CLT-nonhomogeneous-chain}}\label{se:proof3-end}

It only remains to argue that if we set
\begin{equation*}
\eta_i = \E[  d_{n,i}^2 \,|\, \F_{n,i-1}] \quad \text{and} \quad \Delta_n=\sum_{i=1+m}^{n+m} (\eta_i - \E[\eta_i]),
\end{equation*}
then the variance condition \eqref{eq:asymptotic-condition-for-clt-1} implies that
$\Delta_n = o(\Var[S_n])$ in probability
as $n \rightarrow \infty$. We can get this as an easy consequence of the next lemma.

\begin{lemma}[$L^2$-Bound for $\Delta_n$]\label{lm:LLN-conditional-variances-new}
There is a constant $M < \infty$ depending only on $m$ such that for all $n \geq 1$ one has the
inequality
\begin{equation*}
\E[ \Delta_n^2 ] = \Var\big[ \{ \sum_{i=1+m}^{n+m} \E[  d_{n,i}^2 \,|\, \F_{n,i-1}] \} \big]
\leq M C_n^2 \alpha_n^{-2} \Var[S_n].
\end{equation*}
\end{lemma}

\begin{proof} By direct expansion we have
\begin{equation}\label{eq:DeltaExpanded-new}
\E[ \Delta_n^2 ]= \sum_{i = 1+m}^{n+m} \Var[\eta_i]
+ 2 \sum_{i = 1+m}^{n+m}  \E\big[ (\eta_i - \E[\eta_i]) \big\{   \sum_{j = i+1}^{n+m} \big(\eta_j -  \E[\eta_j]\big) \big\}\big],
\end{equation}
and we estimate the two sums separately.
First, by crude bounds and \eqref{eq:Linfty-bound-MDS} we have
$$
\E[  \eta_i ^2 ] \leq  \norm{ \eta_i }_\infty  \, \E[\eta_i] \leq  \norm{d_{n,i}}_\infty^2 \, \E[\eta_i ]
\leq  M C_n^2 \alpha_n^{-2} \E[\eta_i ],
$$
so we obtain that the first sum of  \eqref{eq:DeltaExpanded-new} satisfies the inequality
$$
 \sum_{i = 1+m}^{n+m} \Var[\eta_i] \leq M C_n^2 \alpha_n^{-2} \sum_{i = 1+m}^{n+m}  \E[\eta_i ].
$$
The twin bounds of \eqref{eq:variance-mds-bounds} and the definition $\eta_i = \E[  d_{n,i}^2 \,|\, \F_{n,i-1}]$
then tell us that
\begin{equation}\label{eq:variance-mds-bound-second-lemma}
\Var[S_n] - M C_n^2 \alpha_n^{-2} \leq \sum_{i=1+m}^{n+m} \E[\eta_i] \leq \Var[S_n],
\end{equation}
so we also have the upper bound
\begin{equation}\label{eq:VarBound-new}
 \sum_{i = 1+m}^{n+m} \Var[\eta_i] \leq M C_n^2 \alpha_n^{-2} \Var[S_n].
\end{equation}

To estimate the second sum of \eqref{eq:DeltaExpanded-new},
we first note that $\eta_i$ is $\F_{n,i-1}$-measurable and $\F_{n,i-1} \subseteq \F_{n,i}$,
so, if we condition on $\F_{n,i}$ we have
\begin{equation}\label{eq:decomposition-tower-property-new}
\E\bigg[( \eta_i - \E[\eta_i] )\big\{ \sum_{j = i+1}^{n+m} (\eta_j -  \E[\eta_j]) \big\} \bigg]
  =  \E\bigg[ ( \eta_i - \E[\eta_i] ) \, \E[ \sum_{j = i+1}^{n+m} (\eta_j -  \E[\eta_j]) \,|\, \F_{n,i} ] \bigg]  .
\end{equation}
The definition of $\eta_j$ tells us that $\eta_j-\E[ \eta_j ] = \E [d_{n,j}^2 \, | \, \F_{n, j-1}] - \E [d_{n,j}^2]$ so,
because $\F_{n, i} \subseteq \F_{n, j-1}$ for all $i<j$, one then has
\begin{align*}
\E[ \sum_{j = i+1}^{n+m} (\eta_j -  \E[\eta_j]) \,|\,  \F_{n,i} ]
 & = \sum_{j = i+1}^{n+m} \{ \E [d_{n,j}^2 \, | \, \F_{n, i}] -  \E [d_{n,j}^2]\}.
\end{align*}
These summands have mean zero, so the bound \eqref{eq:OscTrivial2}
and the oscillation inequality \eqref{eq:oscillation-bound-sum-dtsquared} give us
$$
\norm{ \E[ \sum_{j = i+1}^{n+m} (\eta_j -  \E[\eta_j]) |  \F_{n,i} ] }_\infty  \leq  M C_n^2 \alpha_n^{-2}.
$$
When we use this estimate in \eqref{eq:decomposition-tower-property-new}, we see from the non-negativity of $\eta_j$ and the triangle inequality
that
$$
\biggr|\E \bigg[ ( \eta_i - \E[\eta_i] ) \{ \sum_{j = i+1}^{n+m} (\eta_j -  \E[\eta_j]) \} \bigg] \biggr|
\leq M C_n^2 \alpha_n^{-2} \E[\eta_i],
$$
so, after summing over $i \in \{1+m, \ldots, n+m\}$ and recalling the second inequality of \eqref{eq:variance-mds-bound-second-lemma}
we obtain
\begin{equation}\label{eq:bound-second-term-new}
\biggr| \sum_{i = 1+m}^{n+m}  \E\bigg[ ( \eta_i - \E[\eta_i] ) \{   \sum_{j = i+1}^{n+m}  (\eta_j -  \E[\eta_j]) \} \bigg] \biggr|
\leq M C_n^2 \alpha_n^{-2} \Var[S_n].
\end{equation}
By \eqref{eq:DeltaExpanded-new}, the bounds \eqref{eq:VarBound-new} and  \eqref{eq:bound-second-term-new}
complete the proof of the lemma.
\end{proof}

Now, at last, we can use the basic decomposition \eqref{eq:Sn-decomposition} to write
\begin{equation}\label{eq:bottomline}
  \frac{S_n}{\sqrt{\Var[S_n]}} =  \sum_{i=1}^n  \frac{d_{n,i+m}}{\sqrt{\Var[S_n]}} + O\big( \frac{ \norm{V_{n,m}}_\infty}{\sqrt{\Var[S_n]}} \big),
\end{equation}
and it only remains to apply our lemmas.
First, from our hypothesis \eqref{eq:asymptotic-condition-for-clt-1} that $ C_n^2 \alpha_n^{-2} = o(\Var[S_n])$ as $n \rightarrow \infty$,
we see that the $L^\infty$-bound $\norm{d_{n,i}}_\infty \leq M C_n \alpha_n^{-1}$ in Lemma \ref{lm:asymtotic-negligibility}
implies the asymptotic negligibility \eqref{eq:asymtotic-negligibility}
of the scaled differences $d_{n,i+m}/\sqrt{\Var[S_n]}$, $ 1 \leq i \leq n$.
Second, our hypothesis  \eqref{eq:asymptotic-condition-for-clt-1} and the variance bounds \eqref{eq:variance-mds-bounds}
imply the asymptotic equivalence
$$
\Var[S_n] \sim  \sum_{i=1}^n  \E[ d^2_{n,i+m}] \quad \quad \text{ as } n \rightarrow \infty,
$$
so the $L^2$-inequality in Lemma \ref{lm:LLN-conditional-variances-new}
tells us that the weak law \eqref{eq:LLN-conditional-variances} also holds
for the scaled martingale differences.

Taken together, these two observations
imply that the first sum on the right-hand side of \eqref{eq:bottomline} converges in distribution to a standard normal.
Moreover, because of the $L^\infty$-bound
$ \norm{V_{n,m}}_\infty \leq M C_n \alpha_n^{-1}$ given by  \eqref{eq:Linfty-bound-value-functions}, the last term
in \eqref{eq:bottomline} is asymptotically negligible.
In turn, these observations tell us that
$$
\frac{S_n}{\sqrt{\Var[S_n]}} \Longrightarrow N(0,1) \quad \quad \text{as } n \rightarrow \infty,
$$
and the proof of Theorem \ref{th:CLT-nonhomogeneous-chain} is complete.

\section{Dynamic Inventory Management: A Leading Example}\label{se:inventory}

We now consider a classic dynamic inventory management problem where one has $n$ periods and $n$ independent demands
$D_1, D_2, \ldots,D_n$.
We assume that demands all have the same density $\psi$, and that this density has support on a bounded interval contained in $[0,\infty)$.

In each period $1 \leq i \leq n$ one knows the current level of inventory $x$, and the task is to decide
the level of inventory $y \geq x$ that one wants to hold after an order is placed and fulfilled.
Here it is also useful to allow for $x$ to be negative, and, in that case, $|x|$ would represent the level of backlogged demand.
To stay mindful of this possibility, we sometimes call $x$ the generalized inventory level.

We further assume that orders are fulfilled instantaneously
at a cost that is proportional to the ordered quantity; so,
for example, to move the inventory level
from $x$ to $y\geq x$, one places an order of size $y-x$  and incurs a purchase cost equal to $c(y-x)$
where the multiplicative constant $c$ is a parameter of the model.

The model also
takes into account the cost of either holding physical inventory or of managing a backlog.
Specifically, if the current generalized inventory is equal to $x$, then the firm
incurs additional carrying costs that are given by
$$
L(x) =
\begin{cases}
    \hphantom{-}c_h x       & \quad \text{if } x \geq 0 \\
    -c_p x                  & \quad \text{if } x  < 0.
\end{cases}
$$
In other words, if $x\geq 0$, then $L(x)$ represents the cost for holding a quantity $x$ of inventory from one period to the next,
and, if $x< 0$, then $L(x)$ represents the penalty cost for managing a quantity $-x \geq 0$ of unmet demand.

Here we also assume that all unmet demand can be successfully backlogged, so
customers in one period whose demand is incompletely met will return
in successive periods until either their demand has been met or until the decision period $n$ is completed.
If there is still unmet demand at time $n$, then that demand is lost.
Finally, we assume that the purchase cost rate $c$ is strictly smaller than the penalty rate $c_p$,
so it is never optimal to accrue penalty costs when one can place an order. Naturally, the manager's
objective is to minimize the total expected
inventory costs over the decision periods $1, 2, \ldots, n$.

This problem has been widely studied, and, at this point,
its formulation as a dynamic program is well understood --- cf.~\citet{BellmanGlicksbergGross:MS1955},
\citet{Bulinskaya:TPA1964}, or \citet[Section 4.2]{Porteus:SUP2002}.
Specifically, if we let $v_k(x)$ denote the minimal expected inventory cost when there are $k$ time periods remaining and
when $x$ is the current generalized inventory level, then dynamic programming gives us the backwards recursion
\begin{equation}\label{eq:value-functions-inventory}
v_k(x) = \min_{y \geq x} \big\{ c(y-x) + \E[L(y - D_{n-k+1})] + \E[v_{k-1}(y-D_{n-k+1})]\big\},
\end{equation}
for $1 \leq k \leq n$, and
one computes $v_k(x)$ by iteration beginning with $v_0(x) = 0$.

For this model, it is also well-known that
there is a \emph{base-stock policy} that is optimal;
specifically, there are non-decreasing values
\begin{equation}\label{eq:base-stock-level-monotone}
s_1 \leq s_2 \leq \cdots\leq s_n
\end{equation}
such that if the current time is $i$ and the current inventory is $x$, then
the optimal level $\gamma_{n,i}(x)$ at time $i$ for the inventory after restocking is given by
\begin{equation}\label{eq:base-stock-policy}
\gamma_{n,i}(x) =
\begin{cases}
    s_{n-i+1}       & \quad \text{if } x \leq s_{n-i+1} \\
    x                  & \quad \text{if } x > s_{n-i+1}.
\end{cases}
\end{equation}
In other words, if at  time $ i$ the inventory level is below $s_{n-i+1}$ then the optimal action is
to place an order of size $s_{n-i+1} - x$, but if the inventory level is
$s_{n-i+1}$ or higher, then the optimal action is to order nothing. Moreover,
\citet[Theorem 1]{Bulinskaya:TPA1964} also showed that for demands with density $\psi$
and cumulative distribution function $\Psi$,
one has for $n \geq 2$ that
\begin{equation}\label{eq:base-stock-level-bounds}
s_1 = \Psi^{-1} \bigg( \frac{c_p-c}{c_h+c_p}\bigg)  \quad \text{ and }  \quad
s_n \leq s_\infty = \Psi^{-1} \bigg( \frac{ c_p}{c_h+c_p}\bigg).
\end{equation}
These relations will be important for us later.

\subsection*{A CLT for Optimally Managed Inventory Costs}

To begin, we take the generalized inventory at the beginning of period $i = 1$ (before any order is placed)
to be $X_{n,1}=x$, where $x$ can be any element of the interval $[-s_\infty, s_\infty]$.
Subsequently we take $X_{n,i}$ to be
the generalized inventory at the beginning of period $i \in \{2,3,\ldots, n\}$;
so, in view of the base-stock policy \eqref{eq:base-stock-policy}, we have the stochastic recursion
\begin{equation}\label{eq:Xi-inventory}
X_{n,i+1} = \gamma_{n,i}(X_{n,i}) - D_i \quad \quad \text{for all } 1 \leq i \leq n.
\end{equation}
The key point here is that  $\{X_{n,i}: 1 \leq i \leq n+1\}$  is a temporally non-homogenous Markov chain.
Moreover, if the support of the demand density $\psi$ is contained in $[0,J]$ with $0 < J < \infty$ and if $s_1$ and $s_\infty$
are given by \eqref{eq:base-stock-level-bounds},
then by the recursion \eqref{eq:Xi-inventory} we can choose the state space $\X$ of this chain so that
\begin{equation}\label{eq:StateSpaceInclusion}
\X \subseteq [- J, s_\infty].
\end{equation}

Now, if $\pi^*_n$ is the policy that minimizes the total expected inventory cost
that is incurred over $n$ decision periods,
then the total cost that is realized when one follows the policy $\pi^*_n$
is given by
\begin{equation}\label{eq:TotalCost}
\C_n(\pi^*_n) = \sum_{i=1}^n \big\{ c( \gamma_{n,i}(X_{n,i}) - X_{n,i} ) + L (X_{n,i+1})\big\},
\end{equation}
and we see that the total inventory cost $\C_n(\pi^*_n)$ is a special case of the sum \eqref{eq:Sn}.
To spell out the correspondence, we first take $m=1$, and then we take
\begin{equation*}
f_{n,i}(x,y) = c( \gamma_{n,i}(x) - x ) + L (y), \quad \quad \text{for } 1 \leq i \leq n,
\end{equation*}
so finally  \eqref{eq:Xi-inventory} gives us the driving Markov chain.

Theorem \ref{th:CLT-nonhomogeneous-chain} and Corollary \ref{cor:CLT}
now give us a natural path to a central limit theorem for the
realized optimal inventory cost. We only need to isolate a
mild regularity condition on the density function $\psi$ of the demand distribution $\Psi$.

\begin{definition} [Typical Class]\label{m-def}
We say that a probability density function $\psi$ is in the \emph{typical class}
if for each $\epsilon \geq 0$ there is a $\widehat{w}=\widehat{w}(\epsilon)$ such that
\begin{align*}
\psi(w)-\psi(w+\epsilon) &\leq 0 \quad \text{for all } w \leq \widehat{w}, \quad \text{and} \\
\psi(w)-\psi(w+\epsilon) &\geq 0 \quad \text{for all } w \geq \widehat{w}.
\end{align*}
\end{definition}

Densities in the typical class include the uniform density on $[0,J]$, the $\texttt{beta}(\alpha, \beta)$
density with $\alpha\geq 1$ and $\beta\geq 1$, the exponential densities, and the gamma densities. For an
example of a density that is not in the typical class, one can take any density with two separated modes. Such
multi-modal densities are seldom used in demand models.

\begin{theorem}[CLT for Mean-Optimal Inventory Cost]\label{thm:inventory-CLT}
If the demand density $\psi$ is in the typical class and if $\psi$ has bounded support,
then the inventory cost
$\C_n(\pi^*_n)$ realized under the mean-optimal policy $\pi^*_n$
obeys the asymptotic normal law
$$
\frac{\C_n(\pi^*_n) - \E[\C_n(\pi^*_n)]}{\sqrt{\Var[\C_n(\pi^*_n)]}} \Longrightarrow N(0,1), \quad \quad \text{as } n \rightarrow \infty.
$$
\end{theorem}

The one-period cost functions in the sum \eqref{eq:TotalCost} are uniformly bounded because
of the inclusion \eqref{eq:StateSpaceInclusion} and $0<J<\infty$, so
two steps are needed to extract this result from Theorem \ref{th:CLT-nonhomogeneous-chain}.
First we show that the minimal ergodic coefficient of the Markov chain
\eqref{eq:Xi-inventory} is bounded away from zero. Second, we show that
the variance of $\C_n(\pi^*_n)$ goes to infinity as $n\rightarrow \infty$.

After we complete the proof of Theorem \ref{thm:inventory-CLT}, we have two observations.
The first explains why one cannot prove
Theorem \ref{thm:inventory-CLT} by
the device of state space extension and direct invocation of \citeauthor{Dobrushin:TPA1956}'s theorem.
In a nutshell, the issue that if one extends the state space then the coefficient of ergodicity
can become degenerate.
The second observation highlights how one still has the conclusion of
Theorem \ref{thm:inventory-CLT} even for models where there is no immediate fulfillment of placed orders.

\subsection*{A Uniform Lower Bound for the Minimal Ergodic Coefficients}

To establish a uniform lower bound for the minimal ergodic coefficients of
the Markov chain \eqref{eq:Xi-inventory},
we begin with a general lemma which explains the role of the
class of typical densities.

\begin{lemma}[Total Variation Distance Bound]\label{lm:m-TC}
If the density $\psi$  of $D_1$ is in the \emph{typical class}, then for $\epsilon =|\gamma_{n,i}(x') - \gamma_{n,i}(x)|$
one has
\begin{equation}\label{eq:TV-m}
\sup_{B \in \B(\X)} \abs{ \Kn_{i,i+1}(x',B) -\Kn_{i,i+1}(x,B) } = \P(\widehat{w} \leq D_1 \leq \widehat{w}+ \epsilon),
\end{equation}
where $\widehat{w}=\widehat{w}(\epsilon)$ is the value guaranteed by Definition \ref{m-def}.
\end{lemma}

\begin{proof}
Given $x \in \X$ and a Borel set $B\subseteq \X$, we introduce the Borel set
$$
B_x = \gamma_{n,i}(x) -B,
$$
so the transition kernel of the Markov chain \eqref{eq:Xi-inventory} can be written as
$$
\Kn_{i,i+1}(x, B)  = \P (X_{n,i+1} \in B \,|\, X_{n,i}=x ) = \P( D_{1} \in B_x ) = \int_{B_x} \psi(w) \, dw.
$$
Without loss of generality we can assume that $x \leq x'$, so the restocking formula
\eqref{eq:base-stock-policy} gives us
$\gamma_{n,i}(x) \leq \gamma_{n,i}(x')$, and for
$\epsilon = \gamma_{n,i}(x') - \gamma_{n,i}(x) \geq 0$
we find
$$
\Kn_{i,i+1}(x', B)  = \P (X_{n,i+1} \in B \,|\, X_{n,i}=x' ) = \P( D_{1} - \epsilon \in B_x ) = \int_{B_x} \psi(w+\epsilon) \, dw.
$$
The absolute difference in \eqref{eq:TV-m} is then given by
$$
\abs{ \Kn_{i,i+1}(x',B) -\Kn_{i,i+1}(x,B) } = \abs{\int_{B_x} \psi(w) \, dw - \int_{B_x} \psi(w+\epsilon) \, dw },
$$
and the supremum is attained at $B_x^*=\{ w: \psi(w) \geq \psi(w+\epsilon)\}$.
Because $\psi$ belongs to the typical class, Definition \ref{m-def} tells us that the integrals over  $B_x^*$ are equal to
the corresponding integrals over $[\widehat{w}, \infty)$. Hence, we have
\begin{align*}
\sup_{B \in \B(\X)} \abs{ \Kn_{i,i+1}(x',B) & -\Kn_{i,i+1}(x,B) }
= \int_{\widehat{w}}^\infty \big\{\psi(x) -\psi(x+\epsilon) \big\} \, dx \\
& = \P(D_1 \geq  \widehat{w})-\P(D_1-\epsilon \geq \widehat{w}) = \P(\widehat{w}\leq D_1 \leq \widehat{w}+\epsilon),
\end{align*}
just as needed.
\end{proof}

Lemma \ref{lm:m-TC} can be generalized to accommodate multi-modal densities, but since such densities are seldom used as models for demand
distributions, the simple formulation given here covers all the models one is likely to meet in practice.
Moreover, the definitions of $s_1$ and $s_\infty$ given by \eqref{eq:base-stock-level-bounds}
now give us just what we need to make good use of our basic bound \eqref{eq:TV-m}.

\begin{lemma} \label{lm:probability-bound}
For $x,x' \in \X$ and $\epsilon =|\gamma_{n,i}(x') - \gamma_{n,i}(x)|$ one has
\begin{equation*}
\sup_{w \in \R} \P( w \leq D_1 \leq w+\epsilon)
\leq \max \bigg\{\frac{c_p}{c_h+c_p}, \frac{c_h+c}{c_h+c_p} \bigg\}< 1.
\end{equation*}
\end{lemma}

\begin{proof}
Without any loss of generality, we again take $x \leq x'$ and note that the inclusion \eqref{eq:StateSpaceInclusion}
tells us that $x' \leq s_\infty$.
Next, the monotonicity of the restocking formula \eqref{eq:base-stock-policy}
and the defining relations in \eqref{eq:base-stock-level-bounds}
give us that
$$
s_1 \leq \gamma_{n,i}(x) \leq \gamma_{n,i}(x') \leq s_\infty,
$$
so if $\epsilon = \gamma_{n,i}(x') - \gamma_{n,i}(x)$ then one has that
$$
0 \leq \epsilon = \gamma_{n,i}(x') - \gamma_{n,i}(x) \leq s_\infty - s_1.
$$
Now, if $w +\epsilon \leq s_\infty$, then we have the trivial bound
$$
\P(w \leq D_i \leq w+\epsilon) \leq \P(D_i \leq s_\infty),
$$
while if $w +\epsilon \geq s_\infty$ then $w \geq s_1$ and we similarly have
$$
\P(w \leq D_i \leq w+\epsilon) \leq \P(D_i \geq s_1).
$$
By the definitions of $s_1$ and $s_\infty$, we see from \eqref{eq:base-stock-level-bounds}
that
$$
\P(D_i \leq s_\infty) = \frac{c_p}{c_h+c_p}
\quad \quad \text{and} \quad \quad
\P(D_i \geq s_1) = \frac{c_h+c}{c_h+c_p},
$$
where both probabilities are strictly smaller than one because $c < c_p$ and $c_h>0$.
\end{proof}

Our Lemmas \ref{lm:m-TC} and \ref{lm:probability-bound}
tell us that for all  $1 \leq i \leq n$ we have a uniform bound on the contraction coefficient,
$$
\delta ( \Kn_{i,i+1} ) = \sup_{x,x'\in \X} \norm{\Kn_{i,i+1}( x, \cdot ) - \Kn_{i,i+1}( x', \cdot ) }_{\rm TV}
\leq \max \big\{\frac{c_p}{c_h+c_p}, \frac{c_h+c}{c_h+c_p} \big\}.
$$
This tells us that for the minimal ergodic coefficient we have
\begin{equation*}
\alpha_n = \min_{1 \leq i < n} \{ 1 - \delta ( \Kn_{i,i+1} ) \} \geq  \min \big\{\frac{c_h}{c_h+c_p}, \frac{c_p-c}{c_h+c_p} \big\} >0,
\end{equation*}
and this bound completes the
\emph{first step} in the proof of Theorem \ref{thm:inventory-CLT}.

\subsection*{Variance Lower Bound}

Here, as in most stochastic dynamic programs,
the value to-go process \eqref{eq:value-togo}
can be expressed in terms of the value functions that solve the dynamic programming
recursion \eqref{eq:value-functions-inventory}.
In particular, at time $1 \leq i \leq n$, when the current generalized inventory
is $X_{n,i}$ and there are $n-i+1$ demands yet to be realized, one has
$$
V_{n,i} = v_{n-i+1}(X_{n,i}),
$$
where the function $x \mapsto v_{n-i+1}(x)$ is calculated by \eqref{eq:value-functions-inventory}.
Moreover, since we start with $X_{n,1} = x\in \X$, the definition of  $v_n(x)$ gives us
$$
V_{n,1} = v_n(x) = \E[\C_n(\pi^*_n)],
$$
and the martingale decomposition \eqref{eq:Sn-decomposition} can be written more simply as
$$
\C_n(\pi^*_n) - \E[\C_n(\pi^*_n)] = \sum_{i=1}^n d_{n,i+1}.
$$
To bound $\Var[\C_n(\pi^*_n)]$ from below, one then just need to find
an appropriate lower bound on $\E[d_{n,i+1}^2]$ for $1\leq i \leq n$.

For our inventory problem we begin by writing the
martingale differences \eqref{eq:MDS} more explicitly as
\begin{equation}\label{eq:MDS-inventory}
d_{n,i+1} =  c( \gamma_{n,i}(X_{n,i}) - X_{n,i} ) + L (X_{n,i+1}) + v_{n-i}(X_{n,i+1}) - v_{n-i+1}(X_{n,i}).
\end{equation}

Next, we introduce the shorthand $\widehat{v}_{n-i} (x) = L (x) + v_{n-i}(x)$,
and we obtain from the recursion \eqref{eq:value-functions-inventory} and the policy characterization \eqref{eq:base-stock-policy}
that
\begin{align}\label{eq:value-functions-inventory2}
v_{n-i+1}(x)    & =  c( \gamma_{n,i}(x) -x ) + \E[L( \gamma_{n,i}(x) - D_{i})] + \E[v_{n-i}( \gamma_{n,i}(x) -D_{i})]\} \\
                & =  c( \gamma_{n,i}(x) -x ) + \E[\widehat{v}_{n-i} ( \gamma_{n,i}(x) - D_{i}) ].\notag
\end{align}

We now replace $x$ with $X_{n,i}$ in \eqref{eq:value-functions-inventory2}
to get a new expression for $v_{n-i+1}(X_{n,i}),$
and we replace the last summand of \eqref{eq:MDS-inventory} with this expression.
If we recall from \eqref{eq:Xi-inventory} that
$X_{n,i+1} = \gamma_{n,i}(X_{n,i}) - D_i$, then we find after simplification that
$$
d_{n,i+1} =  \widehat{v}_{n-i} (\gamma_{n,i}(X_{n,i}) -D_{i}) - \E[\widehat{v}_{n-i} ( \gamma_{n,i}(X_{n,i}) -D_{i} ) \,|\, \F_{n,i}],
$$
where, just as before, one has $\F_{n,i} = \sigma\{X_{n,1}, X_{n,2}, \ldots, X_{n,i}\}$.
This representation gives us a key starting point for estimating the second moment of $d_{n,i+1}$.

\begin{lemma}
For the inventory cost
$\C_n(\pi^*_n)$ realized under the mean-optimal policy $\pi^*_n$,
there is $\beta>0$ such that, for all $n\geq 1$, one has the variance lower bound
\begin{equation*}
\Var[\C_n(\pi^*_n)] = \sum_{i=1}^n \E[d_{n,i+1}^2] \geq \beta n.
\end{equation*}
\end{lemma}

\begin{proof}
We now let $(D_1',D_2', \ldots,\! D_n')$ be an independent copy of  $(D_1,D_2, \ldots, D_n)$.
Since $X_{n,i}$ is $\F_{n,i}$-measurable,
one then has the further representation
\begin{equation*}
\E[d_{n,i+1}^2\,|\, \F_{n,i} ] = \frac{1}{2} \E[ \{\widehat{v}_{n-i} (\gamma_{n,i}(X_{n,i}) -D_{i}) - \widehat{v}_{n-i} (\gamma_{n,i}(X_{n,i}) -D'_{i})\}^2 \,|\, \F_{n,i} ].
\end{equation*}
Next, we consider the set $G(X_{n,i})$ of all $\omega$ such that
$$
D_{i}(\omega)\in [\gamma_{n,i}(X_{n,i})-s_1,\gamma_{n,i}(X_{n,i})] \quad \text{and} \quad
D'_{i}(\omega) \in [\gamma_{n,i}(X_{n,i})-s_1,\gamma_{n,i}(X_{n,i})].
$$
In other words, at time $i$ when the generalized inventory begins with $X_{n,i}$,
one has for $\omega \in G(X_{n,i})$ that either the demand
$D_{i}(\omega)$ or the demand $D_{i}'(\omega)$ would cause one to order up to the level $s_{n-i}$ in period $i+1$.

If we now replace $i$ with $i+1$ in the recursion \eqref{eq:value-functions-inventory2}
we see that
$$
\{ \widehat{v}_{n-i}(x) - \widehat{v}_{n-i}(y) \}\1\left((x,y) \in [0,s_1]^2\right) = (c + c_h)(y-x)\1\left((x,y) \in [0,s_1]^2\right),
$$
because the two new inventory levels for the next period $i+1$ are both given by $\gamma_{n,i+1}(x) = \gamma_{n,i+1}(y) = s_{n-i}$
and because one incurs holding costs that are proportional to the difference $y-x$.
This last equivalence gives us the lower bound
$$
\E[d_{n,i+1}^2\,|\, \F_{n,i} ]  \geq \frac{1}{2} (c+c_h)^2\E[ \{ D'_i - D_i \}^2 \1(G(X_{n,i})) \,|\, \F_{n,i} ],
$$
and the expectation on the right-hand side is given by
$$
I = \int_{\gamma_{n,i}(X_{n,i})-s_1}^{\gamma_{n,i}(X_{n,i})} \int_{\gamma_{n,i}(X_{n,i})-s_1}^{\gamma_{n,i}(X_{n,i})} \{ u - w \}^2 \psi(u)\psi(w) \,du \, dw.
$$
The integrand is non-negative so we can restrict the domain of integration from $G(X_{n,i})$
to
$$
G'(X_{n,i}) = [\gamma_{n,i}(X_{n,i})-s_1, \gamma_{n,i}(X_{n,i})-\frac{2}{3}s_1] \times [\gamma_{n,i}(X_{n,i})-\frac{1}{3}s_1, \gamma_{n,i}(X_{n,i})]
$$
to obtain the relaxed lower bound
$$
I \geq \int_{\gamma_{n,i}(X_{n,i})-s_1/3}^{\gamma_{n,i}(X_{n,i})} \int_{\gamma_{n,i}(X_{n,i})-s_1}^{\gamma_{n,i}(X_{n,i})-2s_1/3} \{ u - w \}^2 \psi(u)\psi(w) \,du \, dw.
$$
One then has the trivial bound
$$
\frac{s_1}{3} \leq  w - u \quad \quad \text{for all } (u,w) \in G'(X_{n,i}),
$$
so, in the end, we have
$$
I \geq \beta = \frac{s_1^2}{9}
\inf_{w \in [s_1, s_\infty]} \big\{ \Psi(w - \frac{2}{3}s_1) - \Psi(w-s_1) \big\} \big\{ \Psi(w) - \Psi(w-\frac{1}{3}s_1) \big\} > 0.
$$
where the strict positivity of $\beta$ follows from the fact that $\Psi$
is continuous and strictly increasing on the compact set $[0 , s_\infty] \subset [0, J]$.
Thus, the infimum is attained and strictly positive, so in summary we have
$$
\E[d_{n,i+1}^2\,|\, \F_{n,i} ] \geq \beta > 0 \quad \text{for all } 1 \leq i \leq n.
$$
One then completes the proof of the lemma by taking total expectations and summing over $1 \leq i \leq n$.
\end{proof}

\subsection*{State Space Extension: Degeneracy of a Bivariate Chain}

One can write the realized cost \eqref{eq:TotalCost} as an additive functional of a Markov chain
if one moves from the basic chain $\{X_{n,i}: 1 \leq i \leq n+1\}$ on $\X$ to the Markov chain
\begin{equation}\label{eq:Xi-inventory-bivariate}
\{\, \widehat{X}_{n,i}= (X_{n,i}, X_{n,i+1}): 1 \leq i \leq n \, \}
\end{equation}
on the enlarged state space $\X^2 = \X \times \X$. The realized cost \eqref{eq:TotalCost} then becomes
\begin{equation}\label{eq:RealizedCostSum}
\C_n(\pi^*_n) = \sum_{i=1}^n  f_{n,i}( \widehat{X}_{n,i} ),
\end{equation}
and one might hope to apply \citeauthor{Dobrushin:TPA1956}'s CLT  (Theorem \ref{th:CLT-Dobrushin}) to get the asymptotic
distribution of $\C_n(\pi^*_n)$.
To see why this plan does not succeed, one just needs to calculate the minimal ergodic coefficient
for the extended chain \eqref{eq:Xi-inventory-bivariate}.

For any $x,y \in \X$ and any $B\times B'  \in \B(\X^2)$, the
transition kernel of the Markov chain \eqref{eq:Xi-inventory-bivariate} is given by
\begin{align*}
\Kn_{i,i+1}((x,y), B\times B') & = \P (X_{n,i+1} \in B, X_{n,i+2} \in B' \,|\, X_{n,i}=x, X_{n,i+1}=y  )\\
                             & = \1(y \in B) \P( \{\gamma_{n,i+1}(y) - D_{i+1}\} \in B' \,|\, X_{n,i+1}=y  ),
\end{align*}
where $\gamma_{n,i}(x)$ is the function
defined in \eqref{eq:base-stock-policy}.
If we now set $B' = \X$, we have
$$
\Kn_{i,i+1}((x,y), B\times \X) =
\begin{cases}
1 & \quad \text{if } y \in B,\\
0 & \quad \text{if } y \in B^c,
\end{cases}
$$
so for $y \in B$ and $y'\in B^c$ we have
$$
\Kn_{i,i+1}((x,y), B\times \X) - \Kn_{i,i+1}((x,y'), B\times \X) = 1.
$$
This tells us that the minimal ergodic coefficient of the chain \eqref{eq:Xi-inventory-bivariate}
is given by
$$
\alpha_n = 1- \max_{1 \leq i < n} \big\{ \sup_{(x,y),(x',y')} \norm{ \Kn_{i,i+1}((x,y), \,\cdot\, ) - \Kn_{i,i+1}((x',y'),\,\cdot\,) }_{\rm TV}\big\} = 0,
$$
and, as a consequence,  we see that \citeauthor{Dobrushin:TPA1956}'s classic CLT simply does not apply
to the sum \eqref{eq:RealizedCostSum}.

Finally, as one ponders alternative proofs, there is a further possibility that one might consider.
In Section \ref{se:Dobrushin} we noted the possibility of replacing
the minimal ergodic coefficient $\alpha_n$ of the Markov chain \eqref{eq:Xi-inventory-bivariate}
with a potentially less fragile measure of dependence such as the maximal coefficient of correlation $\rho_n$
used by \citet{Pel:PTRF2012}.
For the bivariate chain \eqref{eq:Xi-inventory-bivariate}, the maximal coefficient of correlation is given by
$$
\rho_n = \max_{2 \leq i \leq n} \sup_g \left\{ \frac{\norm{\E[g(\widehat{X}_{n,i}) \,|\, \widehat{X}_{n,i-1}]}_2}{\norm{g(\widehat{X}_{n,i})}_2}
: \norm{g(\widehat{X}_{n,i})}_2 < \infty \text{ and } \E[g(\widehat{X}_{n,i})] = 0\right\},
$$
so for the functional
$$
g(\widehat{X}_{n,i}) = g(X_{n,i}, X_{n,i+1}) = X_{n,i} - \E[X_{n,i}],
$$
one has $\rho_n = 1$, and we see that the CLT of \citet{Pel:PTRF2012} does not help us here.

\subsection*{Accommodation of Lead Times for Deliveries}

To keep the description of the inventory problem as brief as possible, we have assumed
that order fulfillment is instantaneous. Nevertheless, in a more realistic model, one might want to
accommodate the possibility of lead times for delivery fulfillments.

One practical benefit of our ``look-ahead" parameter $m$ is that one can
allow for lead times and still stay within the scope of Theorem \ref{th:CLT-nonhomogeneous-chain}. We do not need to pursue this
particular extension here, but it does help to illustrate another way the look-ahead parameter can be used.

\section{An Application in Combinatorial Optimization: \\ Online Alternating Subsequences}\label{se:alternating}

Given a sequence $y_1, y_2, \ldots, y_n$ of $n$ distinct real numbers,
we say that a subsequence $y_{i_1}, y_{i_2},\ldots, y_{i_k}$, $1\leq i_1 < i_2 < \cdots < i_k \leq n$,  is
\emph{alternating} provided that the relative magnitudes alternate as in
$$
y_{i_1} < y_{i_2} > y_{i_3} < y_{i_4} > \cdots
\quad \text{or} \quad
y_{i_1} > y_{i_2} < y_{i_3} > y_{i_4} < \cdots .
$$
Combinatorial investigations of alternating subsequences go back to
Euler \citep[cf.]{Sta:CM2010}, but probabilistic investigations are more recent;
\citet{Wid:EJC2006}, Pemantle \citep[cf.][p.~568]{Sta:PROC2007}, \citet{Sta:MMJ2008}
and \citet{HouRes:EJC2010} all considered the distribution of the length of the longest alternating
subsequence of a random permutation or of a
sequence $\{Y_1, Y_2, \ldots, Y_n\}$ of independent random variables with the uniform
distribution on $[0,1]$.
There have also been recent applications of this work in
computer science
\citep[e.g.]{Rom:DMTCS2011,BanEpp:AACO2012}
and in tests of independence \citep[cf.][p.~312]{BroDav:SPRI2006}.

Here we consider alternating subsequences in a \emph{sequential}, or \emph{online}, context where we
are presented with the values $Y_1, Y_2, \ldots, Y_n$
one at the time, and the goal is to select an alternating subsequence
\begin{equation}\label{eq:alternating-subsequence}
Y_{\tau_1} < Y_{\tau_2} > Y_{\tau_3} < Y_{\tau_4} > \cdots \lessgtr Y_{\tau_k}
\end{equation}
that has maximal expected length.

A sequence of selection times $1 \leq \tau_1 < \tau_2 < \cdots < \tau_k \leq n$
that satisfy  \eqref{eq:alternating-subsequence} is called
a \emph{feasible policy} if our decision to accept or reject  $Y_i$ as member of the alternating subsequence is based
only on our knowledge of the observations
$\{Y_1, Y_2,\ldots, Y_i\}$.
In more formal terms, the feasibility of a policy is equivalent to requiring that the
indices $\tau_k$, $k=1,2, \ldots$, are all stopping times with respect to
the increasing sequence of $\sigma$-fields $\A_i = \sigma\{Y_1, Y_2,  \ldots, Y_i\}$, $1 \leq i \leq n$.

We now let $\Pi$ denote the set of all feasible policies, and
for $\pi \in \Pi$, we let $A^o_n(\pi)$ be the number of alternating selections made by $\pi$ for the realization
$\{Y_1,Y_2, \ldots, Y_n\}$, so
$$
A^o_n(\pi) = \max\left\{ k : Y_{\tau_1} < Y_{\tau_2} >  \cdots  \lessgtr Y_{\tau_k} \text{ and }
1 \leq \tau_1 < \tau_2 < \cdots < \tau_k \leq n\right\}.
$$
We say that a policy $\pi^*_n \in \Pi$ is optimal (or, more precisely, \emph{mean-optimal}) if
\begin{equation*}
\E[A^o_n(\pi^*_n)] = \sup_{\pi \in \Pi} \E[A^o_n(\pi)].
\end{equation*}
\citet{ArlCheSheSte:JAP2011} found that for each $n$
there is a unique mean-optimal policy $\pi^*_n$ such that
$$
\E[A^o_n(\pi^*_n)] = (2 - \sqrt{2}) n  + O(1),
$$
and it was later found that there is a CLT for $A^o_{n}(\pi^*_n)$.
\begin{theorem}[CLT for Optimal Number of Alternating Selections]\label{th:CLTAlt}
For the mean-optimal number of alternating selections $A^o_{n}(\pi^*_n)$ one has
\begin{equation*}
\frac{A^o_{n}(\pi^*_n) - \E[A^o_{n}(\pi^*_n)]}{\sqrt{\Var[A^o_{n}(\pi^*_n)]}} \Longrightarrow N(0,1)
\quad \quad \text{as }n \rightarrow\infty.
\end{equation*}
\end{theorem}

The main goal of this section is to show that Theorem \ref{th:CLT-nonhomogeneous-chain} leads to a
proof of this theorem that is
quicker, more robust, and more principled than the original proof given in \citet{ArlSte:AAP2014}.
In the process, we also get a second illustration of the
ways in which Theorem \ref{th:CLT-nonhomogeneous-chain} helps one sidestep the degeneracy that sometimes
arises when one tries to use \citeauthor{Dobrushin:TPA1956}'s theorem on a naturally associated bivariate chain.
In fact, it is this feature of \citeauthor{Dobrushin:TPA1956}'s theorem that
initially motivated the development of Theorem \ref{th:CLT-nonhomogeneous-chain}.

\subsection*{Structure of the Additive Process}

To formulate the alternating subsequence problem as an MDP, we first consider a new state space
that consists of pairs $(x,s)$ where $x$ denotes the value of the last selected observation
and where we set $s = 0$ if $x$ is a local minimum and set $s = 1$ if $x$ is a local maximum.
The decision problem then has a notable reflection property:
the optimal expected number of alternating selections that one makes when $k$ observations are yet to be seen
is the same if the system is in state $(x,0)$ or if the system is in state $(1-x, 1)$.
Earlier analyses exploited this symmetry to show that there is a sequence $\{g_k: 1 \leq k < \infty \}$
of optimal threshold functions such that if one sets $X_{n,1}=0$
and lets
\begin{equation}\label{eq:Xi-alternating}
X_{n,i+1} =
\begin{cases}
    X_{n,i} & \text{if $Y_i< g_{n-i+1}(X_{n,i})$}\\
    1 - Y_{i} & \text{if $Y_i\geq g_{n-i+1}(X_{n,i})$},
\end{cases}
\end{equation}
then the optimal number of alternating selections has the representation
\begin{equation*}
A^o_{n}(\pi^*_n) = \sum_{i=1}^{n } \1\left( Y_i\geq g_{n-i+1}(X_{n,i}) \right) = \sum_{i=1}^{n } \1( X_{n,i+1} \neq X_{n,i}).
\end{equation*}

The derivation of these relations requires a substantial amount of work, but for the
purpose of illustrating Theorem \ref{th:CLT-nonhomogeneous-chain} and Corollary \ref{cor:CLT}, one does
not need to go into the details of the construction of these optimal threshold functions.
Here it is enough to note that this representation for $A^o_{n}(\pi^*_n)$ is exactly of the form \eqref{eq:Sn}
that is addressed by Theorem \ref{th:CLT-nonhomogeneous-chain}.

The proof of Theorem \ref{th:CLTAlt} then takes two steps. First, one needs an appropriate lower bound for the
minimal ergodic coefficients of the chain \eqref{eq:Xi-alternating}, and second one needs to check that
the variance of $A^o_{n}(\pi^*_n)$ goes to infinity as
$n\rightarrow \infty$.

The second property is almost baked into the cake, and it is even proved in \citet{ArlSte:AAP2014} that
$\Var[A^o_{n}(\pi^*_n)]$ grows linearly with $n$.
Still, to keep our discussion brief, we will not repeat that
proof. Instead we focus on the new --- and more strategic --- fact that  minimal ergodic coefficients of the Markov chains \eqref{eq:Xi-alternating}
are uniformly bounded away from zero for all $1 \leq i \leq n-2$ and all $n \geq 3$.

\subsection*{A Lower Bound for the Minimal Ergodic Coefficient }

For any $x \in [0,1]$ and any Borel set $B \subseteq [0,1]$,
the Markov chain \eqref{eq:Xi-alternating} has the transition kernel
\begin{align*}
\Kn_{i,i+1}( x, B ) & = \1( x \in B) g_{n-i+1}(x) + \int_{g_{n-i+1}(x)}^1 \1( 1-u \in B) \, du \\
                    & = \1( x \in B) g_{n-i+1}(x) + \abs{ B \cap [0, 1 - g_{n-i+1}(x)]},
\end{align*}
where the first summand of the top equation accounts for the rejection of the newly presented value
$Y_{i}=u$, and the second summand accounts for its acceptance.

To obtain a meaningful estimate for the contraction coefficient of $\Kn_{i,i+1}$
we recall from the earlier analyses that the optimal threshold functions $\{g_k: 1 \leq k <\infty\}$ have the two
basic properties:
\begin{inparaenum}[(i)]
\item\label{eq:gk-identity}  $g_k(x) = x$ for all  $x \in [1/3, 1]$ and all $k \geq 1$, and
\item\label{eq:gk-lowerbound} $g_k(x) \geq 1/6$ for all $x \in [0, 1]$ and all $k \geq 3$.
\end{inparaenum}
Property \eqref{eq:gk-lowerbound} and the recursion \eqref{eq:Xi-alternating}
give us
$X_{n,i} \leq 5/6$ for all $1 \leq i \leq n-2$, and
we see from property \eqref{eq:gk-identity} that
$$
\delta(\Kn_{i,i+1}) = \sup_{x,x'} \norm{\Kn_{i,i+1}( x, \cdot ) - \Kn_{i,i+1}( x', \cdot ) }_{\rm TV} \leq \frac{5}{6}
\quad \text{for all }1 \leq i \leq n-2.
$$
This estimate gives us in turn that
$$
\alpha_{n-2} = \min_{1 \leq i < n-2} \{ 1 - \delta(\Kn_{i,i+1}) \} \geq \frac{1}{6},
$$
so by Corollary \ref{cor:CLT} we have the CLT for $A^o_{n-2}(\pi^*_n)$.
Since $A^o_{n}(\pi^*_n)$ and $A^o_{n-2}(\pi^*_n)$ differ by at most $2$, this also completes the
proof of Theorem \ref{th:CLTAlt}.

\section{A Final Observation} \label{se:conclusions}

Theorem \ref{th:CLT-nonhomogeneous-chain} generalizes the classical CLT of \citet{Dobrushin:TPA1956}, and it offers a pre-packaged
approach to the CLT for
the kinds of additive functionals that one meets in the theory of
finite horizon Markov decision processes. The technology of MDPs is wedded to the pursuit of
policies that maximize total expected rewards,
but such policies may not make good economic sense unless the realized reward
is ``well behaved."
While there are several ways to characterize good behavior,
asymptotic normality of the realized reward is likely to be high on almost anyone's list.
The orientation of Theorem \ref{th:CLT-nonhomogeneous-chain} addresses this issue in a direct and practical way.

The examples of Sections \ref{se:inventory} and \ref{se:alternating} illustrate more concretely
what one needs to do to apply
Theorem \ref{th:CLT-nonhomogeneous-chain}. In a nutshell, one needs to show that the variance of
the total reward goes to infinity and one needs an \emph{a priori} lower bound on the minimal coefficient
of ergodicity. These conditions are not trivial, but, as the examples show, they are not intractable.
Now, whenever one faces the question of a CLT for the total reward of a finite horizon MDP, there is an explicit
agenda that lays out what one needs to do.


\end{document}